\documentclass[11pt,a4paper]{article}
\usepackage{amsmath,mathbbol,amssymb,mathbbol}
\usepackage[T1]{fontenc} \usepackage{mathpazo}
\usepackage[active]{srcltx}

\usepackage{amsmath}
\usepackage{amssymb}
\usepackage{amsbsy}
\usepackage{amsfonts}
\usepackage{graphicx,subfigure,alphalph}
\usepackage{verbatim}
\usepackage{enumitem}
\usepackage{hyperref}
\usepackage{ifthen}
\usepackage{tikz}
\usepackage{psfrag}
\usepackage{color}

 \graphicspath{{./Numerics/}}
 \DeclareGraphicsExtensions{.jpg,.eps,.ps,.pdf,.png}

\RequirePackage{ifthen}
\makeatletter
\newcommand{\logmessage}[1]{\@latex@warning{#1}}
\makeatother
\IfFileExists{../Bibliographie/Users_Guide.txt}{
  }{
  \IfFileExists{../../Bibliographie/Users_Guide.txt}{
    }{
    \IfFileExists{../../../Bibliographie/Users_Guide.txt}{
      }{
        \IfFileExists{../../../../Bibliographie/Users_Guide.txt}{
          }{
            \IfFileExists{../../../../../Bibliographie/Users_Guide.txt}{
              }{
               \IfFileExists{../../../../../../Bibliographie/Users_Guide.txt}{
                 }{
                  \IfFileExists{../../../../../../../Bibliographie/Users_Guide.txt}{
                   }{
        \logmessage{Directory 'Bibliographie' not found}
          }}}}}}}

\def\RR{{\rm I\hspace{-0.50ex}R} }

\def\d{\partial}
\def\O{\Omega}
\def\E{\mathcal E}

\def\R{{\mathbb R}}

 \def\cD{{\mathcal D}}

 \def\E{{\mathcal E}}

 \def\hz{\hat z}

\newcommand{\jump}[1]{\left[ #1\right]}

\DeclareMathOperator{\argmin}{arg\,min}

\newtheorem{theorem}{Theorem}[section]
\newtheorem{lemma}[theorem]{Lemma}

\newtheorem{remark}[theorem]{Remark}

\newenvironment{proof}{{\textbf{Proof.}}}{\hfill \textbf{$\square$}\vspace{0.2cm}}

\title{Application of the $j$-subgradient in a problem of electropermeabilisation}

\author{Zakaria Belhachmi
\thanks{Laboratoire de Math\'ematiques LMIA, Universit\'e de Haute Alsace, 4, rue des Fr\`eres Lumi\`ere,
68096 Mulhouse, FRANCE.(\tt{zakaria.belhachmi@uha.fr})} \and
Ralph Chill
\thanks{
Institut f\"ur Analysis, Fachrichtung Mathematik, TU
Dresden, 01062 Dresden, Germany.(\tt{ralph.chill@tu-dresden.de})}
}

\begin{document}

\maketitle

%\date{\today}

%\maketitle

\begin{abstract}
We study a coupled elliptic-parabolic Poincar\'e-Steklov system arising in electrical cell activity in biological tissues. By using the notion of $j$-subgradient, we show that this system has a gradient structure and thus obtain wellposedness. We further exploit the gradient structure for the discretisation of the problem and provide numerical experiments.
\end{abstract}

\section{Models and problem formulation}

Various problems in fluid mechanics, contact mechanics, heat transfer or diffusion across membranes lead to parabolic or coupled elliptic-parabolic systems of partial differential equations (or inequations) with nonlinear, dynamical conditions prescribed on a Riemannian manifold $\Gamma$ (see \cite{Li69}). 

We consider in this article the problem 
\begin{equation}\label{eq:mod1}
\begin{split} %{cl}
-\Delta u(t,x) & = 0 \quad \text{in } \RR^+ \times (\Omega_i\cup\Omega_e) , \\
\d_t\jump{u}+s(\jump{u}) -\sigma_e \, \d_{n_e} u_e & = 0 \quad \text{on } \RR^+ \times\Gamma ,\\
\jump{\sigma \d_n u} & = 0 \quad \text{on } \RR^+ \times \Gamma , \\
u_i & = g_i \quad \text{on } \RR^+ \times (\d\Omega_i \setminus \Gamma ), \\
u_e & = g_e \quad \text{on } \RR^+ \times (\d\Omega_e \setminus \Gamma ) , \\
u(0,\cdot ) & = u_0 \quad \text{in } \Omega_i\cup\Omega_e .
\end{split}
\end{equation}
Here, $s$ is a given real function, $\Gamma$ is a Lipschitz regular manifold, $\Omega_i$ and $\Omega_e$ are two disjoint, open sets with Lipschitz regular boundary such that
\begin{align*}
 & \Gamma \subseteq \partial \Omega_i \cap \partial\Omega_e ,
\end{align*}
and 
\[
 [u] = u_i|_\Gamma - u_e|_\Gamma 
\]
is the difference of the traces of ${u}_i := {u}|_{\Omega_i}$ and ${u}_e := {u}|_{\Omega_e}$ on the part of the common boundary $\Gamma$. Moreover, $g\in H^1 (\Omega_i\cup\Omega_e )$, and we denote by ${g}_i := {g}|_{\Omega_i}$ and ${g}_e := {g}|_{\Omega_e}$ the restrictions of the function $g$, as well as their traces on $\partial\Omega_i\setminus\Gamma$ and $\partial\Omega_e\setminus\Gamma$, respectively; there will be no danger of confusion when we denote the functions in the interiors and on the boundaries by the same letter. We denote by $n_i$ and $n_e$ the outer normal derivatives at the boundaries of $\Omega_i$ and $\Omega_e$, respectively, and we denote by
\[
 [\sigma \partial_n u] = \sigma_i \partial_{n_i} u_i + \sigma_e \partial_{n_e} u_e = \sigma_i \partial_{n_i} u_i - \sigma_e \partial_{n_i} u_e
\]
the jump of the outer normal derivatives on $\Gamma$; note that $n_i = -n_e$ almost everywhere on $\Gamma$. Here, $\sigma_i$, $\sigma_e >0$ are the (constant) conductivities in $\Omega_i$ and $\Omega_e$, respectively. In the applications which we have in mind, $\Omega_i$ plays the role of an interior domain, $\Omega_e$ is an exterior domain, $\partial\Omega_i = \Gamma$ and $\partial\Omega_e = \Gamma \dot\cup\partial\Omega$. 

When the manifold $\Gamma$ is the external boundary of a set $\Omega$, a gradient system structure has already been identified for similar problems, namely for problems involving the Dirichlet to Neumann-Steklov-Poincar\'e operator. By applying a recent approach from Chill, Hauer \& Kennedy \cite{ChHaKe14}, we identify an abstract gradient system structure for the problem \eqref{eq:mod1}, and thus provide a unified framework to solve it. The point in this approach is that the gradient structure is identified on the boundary space $L^2 (\Gamma )$, where the actual evolution takes place, but we work with an energy defined on $H^1 (\Omega_i\cup\Omega_e)$. We emphasize that in the gradient system framework, a standard and complete theory for wellposedness, regularity, asymptotic behavior, as well as a large choice of efficient numerical methods for computing solutions are well established, and in particular, a large class of steepest descent methods and optimization approaches with well known properties are ready to use. 

Following the seminal work of Hodgkin \& Huxley \cite{HoHu52},  a lot of examples of systems of equations like problem \eqref{eq:mod1} were considered in the study of the electrical cell activity in biological tissues \cite{Fi55, Fi83, Fi81}. 
As a particular example, we consider a revisited version of a model introduced recently by Kavian, Legu\`ebe, Poignard \& Weynans \cite{KLPW14} for the electropermeabilisation (or electroporation) of the membrane of a cell subjected to a short electric pulse. Roughly speaking, under a high transmembrane (electric) potential, the membrane becomes more permeable, thus allowing the diffusion of some molecules; we refer the interested reader to \cite{NeKr99, TeGoRo05, IvViMi10, PePo13, KLPW14} and the references therein for more details on the modelling and the numerous applications of this problem.
In their article, Kavian et al. proposed and analysed a mathematical problem to describe qualitatively the electropermeabilisation for a single cell. They considered a static and a dynamical model with a function $s$ ensuring a smooth transmission between two states of the membrane conductivity. We emphasize that their dynamical model does not fit into our approach limited to autonomous systems like \eqref{eq:mod1} with the function $s$ independent of time, however, we have less restrictive assumptions for $s$, $\Omega_i$ and $\Omega_e$ which enlarge the type of problems for which we can identify the abstract gradient structure. In principle, it is straightforward to generalise the theory to quasilinear equations, for example, for equations where the Laplace operator is replaced by the nonlinear $p$-Laplace operator (see Remark \ref{rem.extensions} below). 

The article is organised as follows. In Section \ref{sec.gradient.structure} we present the theoretical background which leads to the observation that the coupled elliptic-parabolic system \eqref{eq:mod1} is a gradient system. Well-posedness and regularity of solutions then follows from classical results. In Section \ref{sec.discretisation} we discuss the discretisation of the problem \eqref{eq:mod1} relying on the theoretical framework and results obtained in Section \ref{sec.gradient.structure}. In Section \ref{sec.numerics} we present numerical experiments based on the abstract results. We compare the numerical solution with an analytical solution in the context of a simple geometry and a linear transmission law, and provide numerical solutions in the context of nonlinear transmission laws or more complicated geometries.

\section{Gradient structure} \label{sec.gradient.structure}

Before turning our attention to the evolution problem \eqref{eq:mod1}, we consider the stationary problem
\begin{equation} \label{eq.stationary}
 \begin{split}
  -\Delta {u} & = 0 \quad \text{in } \Omega_i \cup \Omega_e , \\
  [\sigma \partial_n {u}] & = 0 \quad \text{on } \Gamma , \\
  s([{u}]) - \sigma_e \, \partial_{n_e} {u}_e & = f  \quad \text{on } \Gamma , \\
  {u}_i & = g_i \quad \text{on } \partial\Omega_i\setminus\Gamma , \\
  {u}_e & = g_e \quad \text{on } \partial\Omega_e\setminus\Gamma , 
 \end{split}
\end{equation}
with a given right-hand side $f\in L^2 (\Gamma )$ and a given function $g\in H^1 (\Omega_i\cup\Omega_e)$. Here again
\[
 [{u}] = {u}_i|_\Gamma - {u}_e|_\Gamma 
\]
is the difference of the traces of ${u}_i := {u}|_{\Omega_i}$ and ${u}_e := {u}|_{\Omega_e}$ on the common part of the boundary $\Gamma$, and $[\sigma \partial_n u] = \sigma_i \partial_{n_i} u_i - \sigma_e \partial_{n_i} u_e$ is the jump of the outer normal derivatives. Let 
\[
 H^1_{0,\Gamma} (\Omega_i\cup\Omega_e ) := \{ u\in H^1 (\Omega_i\cup\Omega_e ) : u_i|_{\partial\Omega_i\setminus\Gamma} = 0 \text{ and }  u_e|_{\partial\Omega_e\setminus\Gamma} = 0 \} .
\]
We say that a function ${u}\in H^1 (\Omega_i \cup\Omega_e )$ is a {\em weak solution} of the stationary problem \eqref{eq.stationary} if $u-g\in H^1_{0,\Gamma} (\Omega_i\cup\Omega_e )$ and, for every ${v}\in H^1_{0,\Gamma} (\Omega_i \cup\Omega_e )$,
\[
 \int_\Omega \sigma \, \nabla {u} \nabla {v} + \int_\Gamma s([{u}]) \, [{v}] = \int_\Gamma f \, [{v}] ,
\]
where $\sigma$ is piecewise constant, namely $\sigma := \sigma_i$ on $\Omega_i$ and $\sigma := \sigma_e$ on $\Omega_e$. Observe that if ${u}$ is a weak solution of the stationary problem, then it satisfies the boundary conditions on $\partial\Omega_i\setminus\Gamma$ and $\partial\Omega_e\setminus\Gamma$ in a weak sense, and
\[
 -\Delta {u} = 0 \text{ in } \cD (\Omega_i \cup \Omega_e )' ,
\]
as one can see by considering test functions ${v}\in\cD (\Omega_i \cup \Omega_e )$ in the definition of a weak solution. Then the Gau{\ss}-Green formula implies, at least if ${u}$ is regular enough, that for every ${v}\in H^1_{0,\Gamma} (\Omega_i \cup \Omega_e )$, 
\begin{align*}
\int_\Gamma f [{v}] & = \int_\Gamma \sigma_i\, \partial_{n_i} {u}_i {v}_i + \int_\Gamma \sigma_e\, \partial_{n_e} {u}_e {v}_e
+ \int_\Gamma s([{u}]) \, [{v}] \\ 
& = \int_\Gamma [\sigma \partial_n {u}] {v}_i - \int_\Gamma \sigma_e\, \partial_{n_e} {u}_e [{v}] 
+ \int_\Gamma s([{u}]) \, [{v}] ,
\end{align*}
and from here one sees that the two remaining boundary conditions on $\Gamma$ are satisfied, too. 

Accordingly, we call a function $u\in L^2_{loc} (\R_+; H^1 (\Omega_i\cup\Omega_e ))$ a {\em weak solution} of the evolution problem \eqref{eq:mod1} if $u-g\in H^1_{0,\Gamma} (\Omega_i\cup\Omega_e )$ for almost every $t\in\R_+$, $[u]\in C(\R_+ ; L^2 (\Gamma ))\cap H^1_{loc} ((0,\infty ); L^2 (\Gamma ))$, $[u]|_{t=0} = u_0$, and for every $v\in H^1_{0,\Gamma} (\Omega_i\cup\Omega_e)$ one has
\[
 \int_\Omega \sigma \, \nabla {u} \nabla {v} + \int_\Gamma s([{u}]) \, [{v}] = - \int_\Gamma \partial_t [{u}] \, [{v}] \text{ for almost every } t\in\R_+ .
\]

As pointed out in the Introduction, we show existence and uniqueness of weak solutions by showing that the evolution problem \eqref{eq:mod1} has a gradient structure.

For this, we follow the approach which has recently been developped in Chill, Hauer \& Kennedy \cite{ChHaKe14} and which is in some sense hidden in the definition of weak solution of the stationary problem or the evolution problem. More precisely, we consider the energy space $V:= H^1 (\Omega_i \cup \Omega_e )$, the reference Hilbert space $H:= L^2 (\Gamma )$, the bounded, linear operator
\begin{align*}
 j: H^1 (\Omega_i \cup \Omega_e) & \to L^2 (\Gamma ) , \\
 {u} & \mapsto [ {u} ] ,
\end{align*}
and the energy $\E : H^1 (\Omega_i \cup \Omega_e ) \to \R\cup\{+\infty\}$ given by
\[
 \E ({u}) = \begin{cases}
             \frac12 \int_\Omega \sigma\, |\nabla {u}|^2 + \int_\Gamma S([{u}] ) & \text{if } u -g \in H^1_{0,\Gamma} (\Omega_i\cup\Omega_e) , \\[2mm]
             +\infty & \text{else} ,
            \end{cases}
\]
where $S$ is a primitive of $s$. For the effective domain one has the equality $D(\E ) = g+H^1_{0,\Gamma} (\Omega_i\cup\Omega_e)$, and the energy is continuously differentiable on this affine subspace  
as one easily verifies. Moreover, $\E$ is globally {\em $j$-quasiconvex} and {\em $j$-quasicoercive} in the sense that the ``shifted'' energy
\begin{align*}
 \E_\omega : H^1_{0,\Gamma} (\Omega_i \cup \Omega_e ) & \to \R , \\
 {u} & \mapsto \E ({u} ) + \frac{\omega}{2} \, \int_\Gamma [{u}]^2 
\end{align*}
is convex and coercive for every $\omega$ large enough; in fact, $\omega >L$ is sufficient, where $L\geq 0$ is the Lipschitz constant of $s$. Recall that coercivity of $\E_\omega$ means that the sublevels $\{\E_\omega \leq c\}$ are bounded for every $c\in\R$; it follows in this special case by an application of the first Poincar\'e inequality. We then define the {\em $j$-subgradient} of $\E$ by
\begin{equation} \label{def.subgradient}
\begin{split}
 \partial_j\E & := \{ (w,f) \in L^2(\Gamma ) \times L^2 (\Gamma ) : \text{there exists } {u}\in D(\E ) \text{ s.t.} \\
& \phantom{:= \{ (w,f) \in L^2(\Gamma )}  w = [{u}] \text{ and for every } {v}\in H^1_{0,\Gamma} (\Omega_i \cup \Omega_e ) \text{ one has} \\
& \phantom{:= \{ (w,f) \in L^2(\Gamma )} \liminf_{t\searrow 0} \frac{\E ({u}+t{v}) - \E ({u})}{t} \geq \int_\Gamma f \, [{v}] \} \\
& = \{ (w,f) \in L^2(\Gamma ) \times L^2 (\Gamma ) : \text{there exists } {u}\in H^1 (\Omega_i \cup \Omega_e ) \text{ s.t.} \\
& \phantom{:= \{ (w,f) \in L^2(\Gamma )} u-g\in H^1_{0,\Gamma} (\Omega_i\cup\Omega_e), \, w = [{u}] \text{, and} \\
& \phantom{:= \{ (w,f) \in L^2(\Gamma )} \text{for every } {v}\in H^1 (\Omega_i \cup \Omega_e ) \text{ one has} \\
& \phantom{:= \{ (w,f) \in L^2(\Gamma )} \int_\Omega \sigma \nabla {u} \nabla {v} + \int_\Gamma s([{u}]) \, [{v}] = \int_\Gamma f \, [{v}] \}
\end{split}
\end{equation}
The equality between the first and the second line follows from the identification of the effective domain, from the fact that $\E$ is continuously differentiable in the affine subspace $D(\E )$, and the special form of its derivative (in fact, G\^{a}teaux differentiable would be sufficient). The following important and at the same time almost trivial lemma is an immediate consequence of the definition of weak solution of the stationary problem \eqref{eq.stationary} and of the definition of the $j$-subgradient. 

\begin{lemma}
One has $(w,f)\in\partial_j\E$ and $w=[{u}]$ as in the definition of $\partial_j\E$, if and only if ${u}$ is a weak solution of the stationary problem \eqref{eq.stationary}. 
\end{lemma}

Note that the definition of the $j$-subgradient differs from the usual variational setting in the sense that the energy is not defined on the space $L^2 (\Gamma )$ itself, so that the $j$-subgradient is not a classical subgradient as defined, for example in \cite{Br73}. Moreover, we are also not in the usual variational setting of a Gelfand triple in which one has, in particular, a dense embedding of the energy space $V$ into the Hilbert space $H = L^2 (\Gamma )$. Our operator $j$ has dense range in $L^2 (\Gamma )$, but it is clearly not injective since the space of test functions on $\Omega_i\cup\Omega_e$ belongs to the kernel of $j$. Note also that the $j$-subgradient may be a multi-valued operator even if the energy on the energy space $V$ is smooth. 

By \cite[Corollary 2.6]{ChHaKe14}, and since the energy is $j$-quasi\-convex and $j$-quasi\-coercive, the $j$-subgradient $\partial_j\E$ is a maximal quasimonotone operator on $L^2 (\Gamma )$, that is, the ``shifted'' operator $\omega I + \partial_j\E$ is maximal monotone on $L^2 (\Gamma )$. Moreover, by \cite[Corollary 2.6]{ChHaKe14} again, the $j$-subgradient is already a subgradient, that is, there exists a quasiconvex, lower semicontinuous functional $\E^H : L^2 (\Gamma ) \to \R \cup \{+\infty \}$ on the reference Hilbert space such that 
\[
 \partial_j\E = \partial\E^H ,
\]
where $\partial\E^H$ is a classical subgradient. Theoretically, \cite[Theorem 2.8]{ChHaKe14} provides a description of this energy defined on $L^2 (\Gamma )$, but this description seems not to be useful for the discretisation considered below. For the purpose of this section, it is only important to know that such a functional $\E^H$ exists. Moreover, by \cite[Theorem 2.8]{ChHaKe14}, the effective domain of the functional $\E^H$ can be characterised as follows:  
\begin{align*}
 D(\E^H ) & := \{ \E^H <+\infty \} \\
& = j (H^1 (\Omega_i\cup \Omega_e )) \\
& = H^\frac12 (\Gamma ) . 
\end{align*}
Here, the second equality is actually \cite[Theorem 2.8]{ChHaKe14}, while the third equality follows from the theory of traces of Sobolev functions \cite{Ad75}. In particular, the effective domain is dense in $L^2 (\Gamma )$, and hence the same is true for the domain of the $j$-subgradient. From these observations we conclude that our system \eqref{eq:mod1} can be rewritten as an abstract, nonautonomous gradient system of the form
\begin{equation} \label{eq.gradient.system}
 \dot w + \partial_j\E (w) \ni f , \quad w(0) = u_0 ,
\end{equation}
where $w := [u]$ is the unknown function from which one has to compute the original solution $u$ by solving, at each time $t$, an elliptic problem. The identification of the effective domain and the classical theory of maximal monotone operators and subgradients of convex, lower semicontinuous energies (see, for example, Brezis \cite[Th\'eor\`emes 3.2, 3.6]{Br73}) yield well-posedness of this problem in the following sense. 

\begin{theorem}[Existence and uniqueness for the abstract gradient system] \label{thm.existence.w}
For every right-hand side $f\in L^2_{loc} (\R_+; L^2 (\Gamma ))$ and every initial value $u_0\in L^2 (\Gamma )$ the gradient system \eqref{eq.gradient.system} admits a unique solution $w\in C(\R_+ ; L^2 (\Gamma )) \cap H^1_{loc} ((0,\infty ); L^2 (\Gamma ))$ and $w(t)\in D(\partial_j\E )$ for almost every $t\in\R_+$. If, in addition, $u_0\in H^\frac12 (\Gamma )$ (and $f\in L^2_{loc} (\R_+;L^2 (\Gamma ))$), then $w\in H^1_{loc} (\R_+;L^2 (\Gamma ))$. Finally, if $u_0\in L^2 (\Gamma )$ and $f=0$, then $w\in C(\R_+ ;L^2 (\Gamma )) \cap W^{1,\infty}_{loc} ((0,\infty ) ;L^2 (\Gamma ))$.
\end{theorem}

\begin{remark}
Strictly speaking, \cite[Th\'eor\`emes 3.2, 3.6]{Br73} only apply to {\em convex}, lower semicontinuous energies, but the proof easily carries over to the case of quasiconvex energies. This is actually true for each of the following methods which may be employed in order to prove the above well-posedness result: the proof by time discretisation (implicit Euler scheme), the proof by space discretisation (the Faedo-Galerkin method), and the proof by Yosida approximations of the subgradient / Moreau-Yosida approximations of the energy, which reduces the gradient system to an ordinary differential equation.
\end{remark}

A lifting yields then that the problem \eqref{eq:mod1} admits for every $u_0\in L^2 (\Gamma )$ a unique weak solution, and this weak solution has the regularity described above.

\begin{theorem}[Existence and uniqueness of solutions of weak solutions of \eqref{eq:mod1}]
For every initial value $u_0\in L^2 (\Gamma )$ the problem \eqref{eq:mod1} admits a unique weak solution $u\in L^2_{loc} (\R_+ ; H^1 (\Omega_i \cup \Omega_e ))$.
\end{theorem}

\begin{proof}
By Theorem \ref{thm.existence.w}, we already have the existence of a solution $w\in W^{1,\infty}_{loc} ((0,\infty );L^2 (\Gamma ))$ of the abstract gradient system \eqref{eq.gradient.system}. Choose $\omega\in\R$ large enough such that $\E_\omega$ is convex and coercive. Then the differential inclusion in \eqref{eq.gradient.system} (with $f=0$) can be rewritten as 
\[
 \omega w + \partial_j\E (w) \ni \omega w - \dot w .
\]
One easily verifies that $\omega w + \partial_j\E (w) = \partial_j \E_\omega (w)$, where, as before, $\E_\omega$ is the shifted energy functional. By definition of the subgradient (see \eqref{def.subgradient}), and by the convexity of $\E_\omega$, we have
\begin{align*}
 \partial_j\E_\omega & = \{ (w,f) \in L^2(\Gamma ) \times L^2 (\Gamma ) : \text{there exists } {u}\in D(\E_\omega ) = D(\E ) \text{ s.t.} \\
& \phantom{:= \{ (w,f) \in L^2(\Gamma )}  w = [{u}] \text{ and for every } {v}\in H^1 (\Omega_i \cup \Omega_e ) \text{ one has} \\
& \phantom{:= \{ (w,f) \in L^2(\Gamma )} \E_\omega ({u}+{v}) - \E_\omega ({u}) \geq \int_\Gamma f \, [{v}] \} \\
& = \{ (w,f) \in L^2(\Gamma ) \times L^2 (\Gamma ) : \text{there exists } {u}\in D(\E_\omega ) = D(\E ) \text{ s.t.} \\
& \phantom{:= \{ (w,f) \in L^2(\Gamma )}  w = [{u}] \text{ and for every } {v}\in H^1 (\Omega_i \cup \Omega_e ) \text{ one has} \\
& \phantom{:= \{ (w,f) \in L^2(\Gamma )} \E_\omega ({u}+{v}) - \int_\Gamma f \, [u+v] \geq \E_\omega ({u}) - \int_\Gamma f \, [{u}] \} .
\end{align*}
As a consequence of this identification, if $(w,f)\in \partial_j\E_\omega$, then there exists $u\in H^1 (\Omega_i\cup\Omega_e)$ such that $u-g\in H^1_{0,\Gamma} (\Omega_i\cup\Omega_e)$ and 
\begin{equation} \label{eq.argmin}
 u = \argmin \, (\E_\omega (v) - \int_\Gamma f [v] ) .
\end{equation}
By choosing $\omega$ even larger, if necessary, we see from the special form of the energy $\E$ that the function
\begin{align*}
 H^1 (\Omega_i\cup\Omega_e ) & \to \R\cup\{+\infty\} , \\
 v & \mapsto \E (v) + \frac{\omega}{2} \int_\Gamma [v]^2 - \int_\Gamma f [v]
\end{align*}
is {\em strictly} convex. Hence, the minimizer in \eqref{eq.argmin} is uniquely determined. Standard arguments for classical subgradients and inverses of strictly monotone operators yield that there exists a constant $C\geq 0$ such that for any pair $u_1$, $u_2\in H^1 (\Omega_i\cup\Omega_e )$ of solutions of the minimisation problem \eqref{eq.argmin} for given functions $f_1$, $f_2\in L^2 (\Gamma )$ one has
\[
 \| u_1 -u_2\|_{H^1 (\Omega_i\cup\Omega_e)} \leq C\, \| f_1-f_2\|_{L^2 (\Gamma )} .
\]
Applying these observations on the differential inclusion above, we find that there exists a unique $u\in L^2 (0,T; H^1 (\Omega_i\cup\Omega_e))$ such that $u-g\in H^1_{0,\Gamma} (\Omega_i\cup\Omega_e)$ and $[u] = w$ almost everywhere. By construction, $u$ is the unique weak solution of \eqref{eq:mod1}.
\end{proof}

\begin{remark} \label{rem.extensions}
 We repeat that our approach to proving well-posedness of the system \eqref{eq:mod1} is formally restricted to the case when the energy does not depend on time, but the framework we are working in allows us to consider several possible generalisations.

(a) The theory works in the same way if we choose $s$ to be a function of the form $s=s_0 +s_1$, where $s_0$ is monotone (nondecreasing) and $s_1$ is globally Lipschitz continuous. The energy $\E$ is defined in the same way, with a primitive of $S$, but its effective domain is in general no longer an affine subspace, at least if $s_0$ has superlinear growth. In this case $\E$ is no longer G\^ateaux differentiable on $g+H^1_{0,\Gamma} (\Omega_i\cup\Omega_e )$, but merely lower semicontinuous. The $j$-subgradient $\partial_j\E$ is then only defined by the first line in \eqref{def.subgradient}. However, the energy will still be quasiconvex and quasicoercive, so that the abstract problem \eqref{eq.gradient.system} is still well-posed in the sense described above.

(b) Similarly, like in the case of the Dirichlet-to-Neumann operator considered in \cite{ArEl12, AEKS14} (linear case) and \cite{ChHaKe14} (nonlinear case), the regularity assumptions on $\Omega_i$ and $\Omega_e$ may be considerably relaxed. It suffices to assume that $\partial\Omega_i$, $\partial\Omega_e$ and $\Gamma$ have locally finite $(d-1)$-dimensional Hausdorff measure. Traces are then to be understood in a weaker sense; see \cite{Da00, ChHaKe14} for the definition which goes back to Mazya \cite{Mz85}.

(c) The method shows that the Laplace operator may be replaced by the $p$-Laplace operator or any other nonlinear elliptic operator with variational structure. This might be of importance if in the applications described in the Introduction it becomes necessary to consider a larger class of models with nonlinear diffusion operators. In the present work we shall show some numerical experiments, and we have therefore restricted ourselves to the case of semilinear problems with the Laplace operator as leading operator.
\end{remark}

\section{Discretisation} \label{sec.discretisation}

In this section we propose to find approximate solutions of the problem \eqref{eq:mod1} by using a semi-discrete implicit time scheme, that is, given a time step $h>0$, we are seeking a sequence $(z^n)_{n=0}^{[T/h]}$, thought to be an approximation of $(u(nh))_{n=0}^{[T/h]}$, where $u$ is a solution of \eqref{eq:mod1}. More precisely, $(z^n)$ is a solution of the discrete system
\begin{align*}
 & \frac{z^{n+1}-z^n}{h}+\d_j\E(z^{n+1})\ni 0 , \\
 & z^0 = u_{0,h} .
\end{align*}
Recalling that $\partial_j\E$ is actually a subgradient of some energy $\E^H$ defined on $L^2 (\Gamma )$, it is well known that this system is equivalent to solving in each step a minimisation problem, and so we obtain the so called {\em proximal algorithm} \cite{BaCo11,BoVa04,Le89}:
\begin{align*}
 z^0 & = u_{0,h} , \\
 z^{n+1}  & =\argmin\, (\E^H(w)+\frac{1}{2h}\Vert w-z^n\Vert_{L^2(\Gamma )}^2) \\
  & = \argmin\, \inf_{[u]=w}(\E(u)+\frac{1}{2h}\Vert [u]-z^n\Vert_{L^2(\Gamma )}^2), 
\end{align*}
where in the last inequality we have used an identification of $\E^H$ from \cite[Corollary 2.9]{ChHaKe14}. Thus, instead of solving a minimisation problem for the energy $\E^H$, which is difficult to identify or to handle in practical situations, we solve the modified proximal algorithm
\begin{equation}\label{eq:var0}
\begin{split}
 z^0 & = u_{0,h} , \\
\hat z^{n+1} & = \argmin\, (\E(u)+\frac{1}{2h}\Vert [u]-z^n\Vert_{L^2 (\Gamma )}^2), \\
 z^{n+1} & = [ \hat{z}^{n+1} ],
\end{split}
\end{equation}
where now the minimisation is performed for the energy $\E$ in the reference energy space $H^1 (\Omega_i \cup\Omega_e )$ (the effective domain of $\E$). This energy is explicitly given, but we have to pay a price by passing from a minimisation problem in the space $L^2 (\Gamma )$ to a minimisation problem in the reference energy space $H^1 (\Omega_i \cup\Omega_e )$, that is, from a function space over $\Gamma$ to a function space over $\Omega_i\cup\Omega_e$, which adds one space dimension in the domain. However, at the same time, the structure of the problem \eqref{eq:mod1}, which couples a parabolic equation on $\Gamma$ with an elliptic equation in $\Omega_i\cup\Omega_e$, suggests that it is necessary to pass through $\Omega_i\cup\Omega_e$ anyhow.

\begin{remark}
In the case of the example considered below, it is possible to express the problem \eqref{eq:mod1} on the manifold $\Gamma$ by
\begin{equation} \label{eq:absform}
\dot{U}+\Lambda_\sigma U + {\mathbf S}(U)=0,\quad U(0)=U_0,
\end{equation}
with $U=(u_i|_{\Gamma},u_e|_{\Gamma})$, 
\[
\Lambda_\sigma=\begin{pmatrix}
\Lambda_{\sigma_i} & 0 \\
0 & \Lambda_{\sigma_e} 
\end{pmatrix}\quad \text{ and } \quad \mathbf{S}(U)=\begin{pmatrix}
s(u_i-u_e) \\
-s(u_i-u_e)
\end{pmatrix} ,
\]
where $\Lambda_{\sigma_i}$ and $\Lambda_{\sigma_e}$ denote appropriate Dirichlet-to-Neumann operators on $\Gamma$. When the geometry is simple (that is, for example, when $\Omega_i$ and $\Omega_i\cup\Omega_e\cup\Gamma$ are concentric balls) and when the diffusion coefficients $\sigma_i$ and $\sigma_e$ are constant, these operators are easy to compute (the first one admits in fact an explicit representation \cite[Section 36.2]{La02}), and one might solve 
the gradient system directly on $\Gamma$. However, such geometries seem not realistic for cells and biological tissues. That is why we prefer to have a more general approach for solving problem \eqref{eq:mod1}.
\end{remark}

The existence and uniqueness for the problem \eqref{eq:var0} is well known, at least if the time step $h$ is small enough ($h<\frac{1}{L}$ is sufficient, where $L$ is the Lipschitz constant of $s$), and the sequence $(z^n)_n$ is then well defined. Note that the variational Euler-Lagrange equation corresponding to \eqref{eq:var0} is
\begin{equation}\label{eq:var1}
\int_{\Omega_i\cup\Omega_e} \sigma \nabla\hz^{n+1} \nabla v +\int_\Gamma s([\hz^{n+1}])\,\jump{v}\,d\sigma+\int_\Gamma\,\frac{1}{h}([\hz^{n+1}]-z^n)\,\jump{v}\,d\sigma=0.
\end{equation}
Thus the algorithm reads as follows:
\begin{itemize}
\item[-]  Choose $z^0$ ($=u_{0,h}$), an approximation of the exact initial value $u_0$.
\item[-]  Given $z^n\in H^\frac12 (\Gamma )$, compute ${\hat z}^{n+1}$, solution of \eqref{eq:var0} or, equivalently, \eqref{eq:var1}.
\item[-]  Set $z^{n+1} := [\hat z^{n+1}]$.
\end{itemize}
Note that $z^0$ is any element in the closure of $j(V) = H^\frac12 (\Gamma )$, that is, $z^0 \in L^2 (\Gamma )$, and after one iteration $(z^n)_n$ remains in $H^\frac12 (\Gamma )$,  the effective domain of $\E^H$. We emphasize that the gradient structure of the system \eqref{eq:var0} allows one to use any optimization method to solve the minimisation step. However, since the reference energy space $H^1(\Omega_i\cup\Omega_e)$ contains functions with a jump on the manifold $\Gamma$, a natural approach might be based on an alternating algorithm of minimisation in the sub-domains. More precisely,  the method consists of a non overlapping Schwarz algorithm to solve the problem \eqref{eq:var1}. For each time step $n$, and given $z^n$, we denote $z^{n+1}$ by 
$u_i^{n+1}-u_e^{n+1}$ and we drop the index $n+1$ for simplicity. Then, we compute a sequence $(u^k)_k$
in the following way: given $u^k$, we solve
\begin{equation}\label{eq:sch1}
\begin{cases}
-\Delta {u}_i^{k+1}=0  &  \text{in } \Omega_i , \\[2mm]
\frac{{u}_i^{k+1}-{u}_e^{k}-z^n}{h}+s(({u}_i^{k+1}-{u}_e^{k}))+\sigma_i \, \partial_{n_i} {u}_i^{k+1}=0 & \text{on } \Gamma , \\[2mm]
 u_i = g_i & \text{on } \partial\Omega_i \setminus \Gamma 
\end{cases}
\end{equation}
\begin{equation}\label{eq:sch2}
\begin{cases}
-\Delta {u}_e^{k+1}=0  & \text{in } \Omega_e , \\[2mm]
\frac{{u}_i^{k^*}-{u}_e^{k+1}-z^n}{h}+s(({u}_i^{k^*}-{u}_e^{k+1}))+ \sigma_e\, \partial_{n_e} {u}_e^{k+1}=0 & \text{on } \Gamma , \\[2mm]
u_e^{k+1}=g_e & \text{on } \partial\Omega_e\setminus\Gamma ,
\end{cases}
\end{equation}
with $k^*=k$ or $k^*=k+1$.
The existence  of solutions for the sub-problems \eqref{eq:sch1}-\eqref{eq:sch2} follows from the assumptions on $s$ and the condition $h<\frac{1}{L}$. The convergence of the Schwarz algorithm with nonlinear transmission conditions is not obvious and is beyond the scope of this paper.  We emphasize that several choices on the coupling terms on $\Gamma$ are possible, for example nonlinear coupling terms which are both implicit for the interior and the exterior domain, nonlinear coupling terms which are implicit in one of the domains, and nonlinear coupling terms which are both explicit for the interior and the exterior domain.

\subsection*{A remark on a linear version of the algorithm} 
A  variant of the Schwarz algorithm consists in linearizing the transmission conditions. For this, we set $s(\left[ u\right])=a(\left[ u\right])\left[ u\right]$ (in particular, we assume $s(0) = 0$, which is a reasonable assumption). Then we can rewrite the internal sub-problem \eqref{eq:sch1} as
\begin{equation}\label{eq:sch1l}
\begin{cases}
-\Delta {u}_i^{k+1}=0,  & \text{in } \Omega_i , \\[2mm]
\frac{{u}_i^{k+1}-{u}_e^{k}-z^n}{h}+a({u}_i^{k}-{u}_e^{k})({u}_i^{k+1}-{u}_e^{k})+\sigma_i\,\partial_{n_i} {u}_i^{k+1}=0 & \text{on } \Gamma , \\[2mm]
u_i^{k+1} = g_i & \text{on } \partial\Omega_i\setminus\Gamma .
\end{cases}
\end{equation} 
If we set $a_k:=a(\left[ {u}^k\right]) = a({u}_i^{k}-{u}_e^{k})$ and 
\[
B_{k}(u)= (\frac{1}{h}+a_k)\,u ,
\]
then we may rewrite the nonlinear Schwarz algorithm as a linear implicit method of the form 
\begin{equation}\label{eq:sch1ll}
\begin{cases}
-\Delta {u}_i^{k+1}=0  & \text{in } \Omega_i ,\\[2mm]
B_{k}({u}_i^{k+1}) + \sigma_i\, \partial_{n_i} u_i^{k+1} = B_{k} ({u}_e^k ) + \frac{z^n}{h} & \text{on } \Gamma , \\
u_i^{k+1} = g_i & \text{on } \partial\Omega_i\setminus\Gamma ,  
\end{cases}
\end{equation}
\begin{equation}\label{eq:sch2ll}
\begin{cases}
-\Delta {u}_e^{k+1}=0  & \text{in } \Omega_e ,\\[2mm]
B_{k}({u}_e^{k+1}) + \sigma_e\, \partial_{n_e} u_e^{k+1} = B_{k} ({u}_i^k ) + \frac{z^n}{h} & \text{on } \Gamma  , \\[2mm]
{u}_e^{k+1}=g_e & \text{on } \partial\Omega_e \setminus\Gamma .\\
\end{cases}
\end{equation}
For $\epsilon>0$, we set this Schwarz method under the form of the  linear transmission Robin condition
\begin{equation}\label{eq:sch1lleps}
\begin{cases}
-\Delta {u}_i^{\epsilon,k+1}=0  & \text{in } \Omega_i,\\[2mm]
B_{k}({u}_i^{\epsilon,k+1}) + \sigma_i\, \partial_{n_i} u_i^{k+1} = B_{k} ({u}_e^{\epsilon,k}) + \epsilon \partial_{n_e} u_e^{\epsilon,k}& \text{on } \Gamma, \\
 u_i^{k+1} = g_i & \text{on } \partial\Omega_i\setminus\Gamma ,  
\end{cases}
\end{equation}
\begin{equation}\label{eq:sch2lleps}
\begin{cases}
-\Delta {u}_e^{\epsilon,k+1}=0  & \text{in } \Omega_e ,\\[2mm]
B_{k}({u}_e^{\epsilon,k+1}) + \sigma_e\, \partial_{n_e} u_e^{\epsilon,k+1} = B_{k} ({u}_i^{\epsilon,k} ) + \epsilon \partial_{n_i} u_i^{\epsilon,k} & \text{on } \Gamma  , \\[2mm]
{u}_e^{\epsilon,k+1}=g_e & \text{on } \partial\Omega_e\setminus\Gamma .
\end{cases}
\end{equation}
Note that this is a slight generalization of the Schwarz method considered in \cite[Theorem 1, and the section V]{Li88I} and the convergence of this algorithm may be obtained following the same lines. In particular, for general geometries 
and domain decompositions, or for non-convex energies $\E$, the linearization of the algorithm might be suitable.

\begin{remark}
 The Schwarz method is not the unique possible choice to solve problem \eqref{eq:var1}, but it is a quite natural approach. In fact, for many classical problems (for example, domain decomposition), the Schwarz method is an elegant approach, although it may have some shortcomings such as expansive cost or slow convergence. When it is used with the state-of-the-art scientific computing methods (parallel programming, preconditioning), it becomes a very attractive tool \cite{GlPaPe95}. For the problem considered here, it is feasible even for more than one cell, for example, a network of cells.
 \end{remark}

\section{Numerics} \label{sec.numerics}

In this section we consider three examples to test our approach. We emphasize that our numerical simulations are presented as a proof of the concept rather than the results of an optimized computing code for solving general problems of j-gradient type. In particular, we do not choose the physical parameters for the model of electropermeabilisation and do not try to make any comparison with existing models. The computations are done on a laptop mac-pro i5 (2.5 GHz) with the open source software FreeFem++ \cite{He12}. We use the nonlinear algorithm and the Schwarz iterations are performed with the nonlinear optimization library IPopt \cite{WaBi06}. The first example treats a simple geometry of the cell and linear transmission conditions on $\Gamma$ where actually an analytic solution is available (see \cite{KLPW14}); we may thus compare the analytical and the numerical solution. The second and the third examples treat more complex transmission conditions at the membrane $\Gamma$, namely a nonlinear, monotone transmission law proposed by Kavian et al. \cite{KLPW14}, and one condition of a double well type. The two nonlinearities are of a rather different nature and might serve as representatives of various other transmission conditions. We recall that in our approach several generalizations are possible and we end up this Section with some non trivial geometries.

\subsection{Example 1} 

In our first example we let $0<R_1<R_2$ and put $\Omega_i := B(0,R_1)$, $\Omega_e := B(0,R_2)\setminus \overline{B(0,R_1)}$, and $\Gamma := \partial B(0,R_1)$, that is, $\Omega_i$ is the disk of radius $R_1$, $\Omega_e$ is a concentric annulus with radii $R_1$ and $R_2$, and $\Gamma$ is the circle of radius $R_1$. We assume given two constant conductivities, $\sigma_i$ in $\Omega_i$ and $\sigma_e$ in $\Omega_e$, respectively, and a Dirichlet boundary condition $g=E\,R_2\cos(\theta)$ on $\partial B(0,R_2)$, where $E$ is a given constant electrical field intensity. The function $s$ is assumed to be linear, that is, $s(\lambda)=S_L\cdot \lambda$, where $S_L$ is a constant. An explicit solution for these data is given in \cite{KLPW14} in polar coordinates, namely
\begin{align*}
& u(r,\theta)=(\alpha_e\,r+b_er^{-1})\cos(\theta) && \text{ for } (r,\theta)\in\left[ R_1,R_2\right]\times\left[ 0,2\pi\right] , \text{ and} \\
& u(r,\theta)=\alpha_i\,r\cos(\theta) && \text{ for } (r,\theta)\in\left[ 0,R_1\right]\times\left[ 0,2\pi\right] ,
\end{align*}
where, if we set $A=\frac{1}{2}(\frac{\sigma_i}{S_LR_1}+1+\frac{\sigma_i}{\sigma_e})$ and $B=(\frac{\sigma_i}{S_LR_1}+1-\frac{\sigma_i}{\sigma_e})$,
\[
\alpha_e=A\alpha_i,\quad  \beta_e=B\alpha_iR_1^2,\quad \alpha_i=\frac{E}{(A+B(\frac{R_1}{R_2})^2)}.
\]
For the simulation we take $\sigma_i=\sigma_e=1$ and $R_1=1$, $R_2=2$. 

In Figure \ref{fig1}, we plot the convergence curve of the $L^2$-error of the solution at the final time $T=1$, as a function of the space discretization parameter $h_x$ in a log log scale and a fixed time step $h=0.1$. 

\begin{figure}[h!]\label{fig1}
\begin{center}
{\includegraphics[width = 0.7\textwidth]{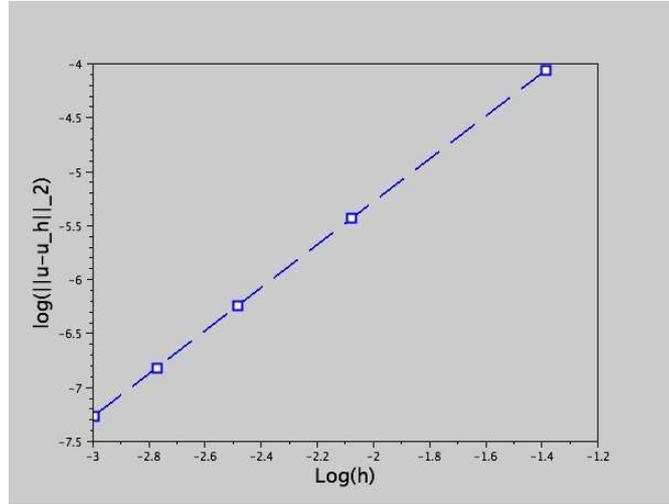}}
\end{center}
\caption{Convergence curve for the $L^2$ error in log log scale. $S_L=10^8$, rate of convergence $1.97$}
\end{figure}

We note that the algorithm converges very quickly in this example and the solution is accurately computed, as formally expected from theoretical considerations. This supports that the Schwarz method should converge even with nonlinear transmission conditions. Moreover, as we work in a variational setting, we may consider more general geometries and boundary conditions without supplementary efforts. 

Note that the convergence rate, in this example,  decreases with $S_L$ whatever the mesh size $h_x$ is and for a fixed time step $h$. This might be justified by the fact that when $S_L$ decreases, the solution becomes more singular. In addition,  for smaller $S_L$, the time step should be chosen small, too, to ensure the coerciveness of the energy.
 
\begin{remark}
For the electropermeabilisation problem, the dynamical transmission condition is
\[
C_m\partial_t {u}+s_m(u)+\sigma_i\d_{n_i} {u}_i=0 \text{ on } \Gamma ,
\]
where $\sigma_i$ is the internal constant conductivity (a typical value is $0.455$ S/m) and $C_m$ is the capacitance (a typical value is $9.5 10^{-3}\ F/m^2$ \cite{NeKr99}). We have not taken exactly these values in this example, because we are only interested in the qualitative behaviour of the system. Nevertheless, the large difference between $S_L$ and $S_R$ allows for an optimal rate of convergence of the algorithm.
\end{remark}

\subsection{Example 2}

In the second example, we choose $\Omega_i$, $\Omega_e$ and $\Gamma$ as in the example 1, and we consider the nonlinear function $s$ to be the derivative of a double well potential with equilibrium points $S_L<S_a<S_R$, that is,
\[
s(t)=-\epsilon^2\,A_m\,(t-S_R)(t-S_a)(t-S_L) , 
\]
with $\epsilon>0$, $A_m\geq 0$. We assume that $S_a< \frac{S_L+S_R}{2}$. This is a particular example for a more general choice of functions satisfying 
\[
s'(S_R)<0,\quad s'(S_a)>0,\quad s'(S_L)<0.
\]
In this example we consider the same boundary conditions as in the first example and a zero initial condition. In Figures \ref{fig2-1}-\ref{fig2-6} we plot the solution $u=(u_i,u_e)$ at times $T=0.5$ and $T=1.0$. The time step is $0.05$ and the mesh size $h_x= 0.07$.
In this example, we have set $S_L=1.9$, $S_R=10^2$, $S_a=10$, $Am=1$ and $\epsilon=10^{-3}$. Note that the colormap is calculated for each single image so that the colormap for the solution in the entire domain $\Omega_i\cup\Omega_e$ does not necessarily appear to be the sum of the colormaps of the two images in each sub-domain (compare, for example, \ref{fig2-4}, \ref{fig2-5}, and \ref{fig2-6}).

\begin{remark}
Note that as we may expect the solutions to be very smooth except in a neighborhood of $\Gamma$, we may use different meshes on 
$\O_i$ and $\O_e$ and refine the meshes close to $\Gamma$ (see example 3).
\end{remark}

\begin{figure}[h!]
\begin{center}
\subfigure[The solution in $\Omega_i$  \label{fig2-1}]
{\includegraphics[width = 0.32\textwidth]{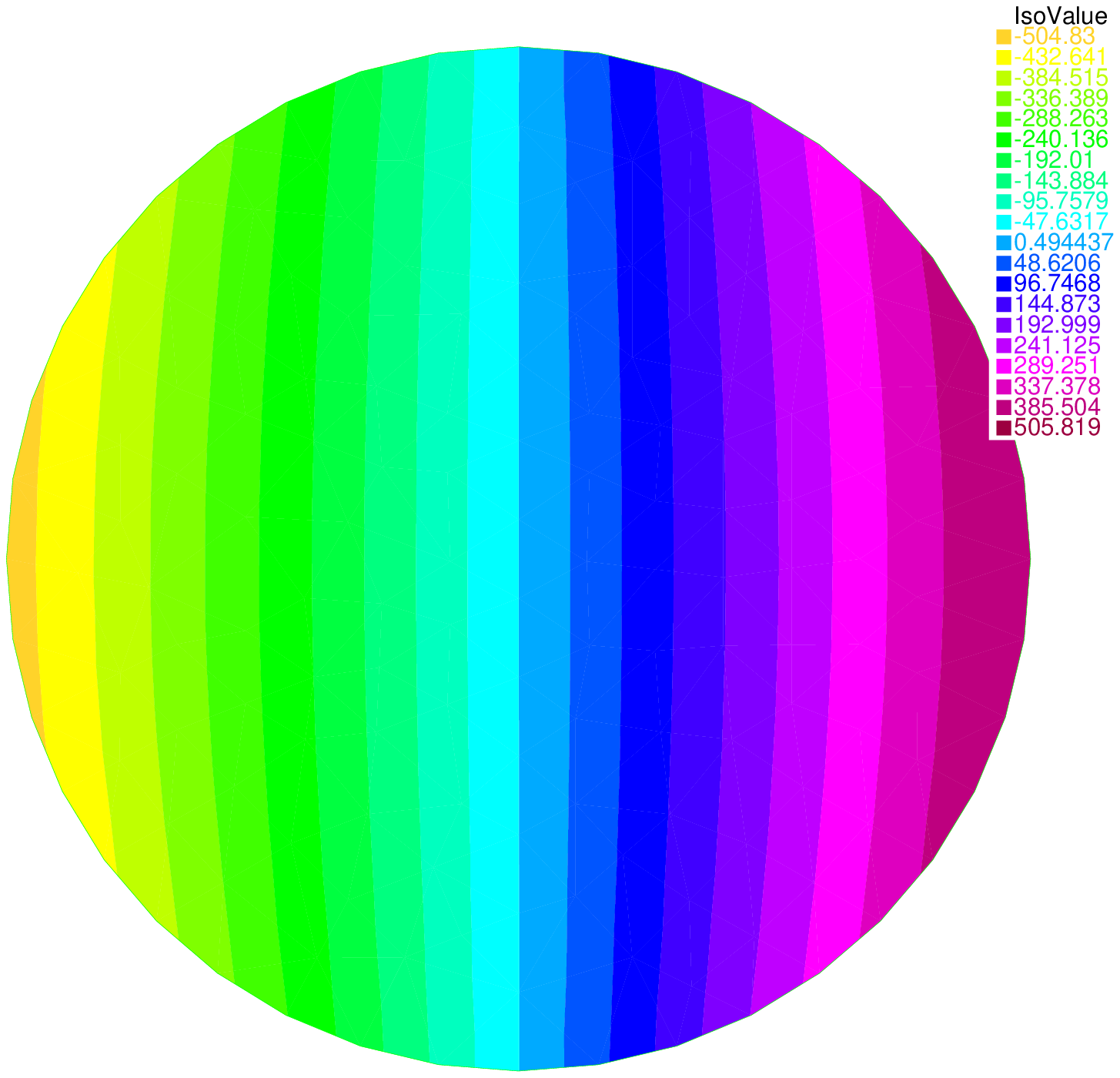}}
\subfigure[The solution in $\Omega_e$ \label{fig2-2}]
{\includegraphics[width = 0.32\textwidth]{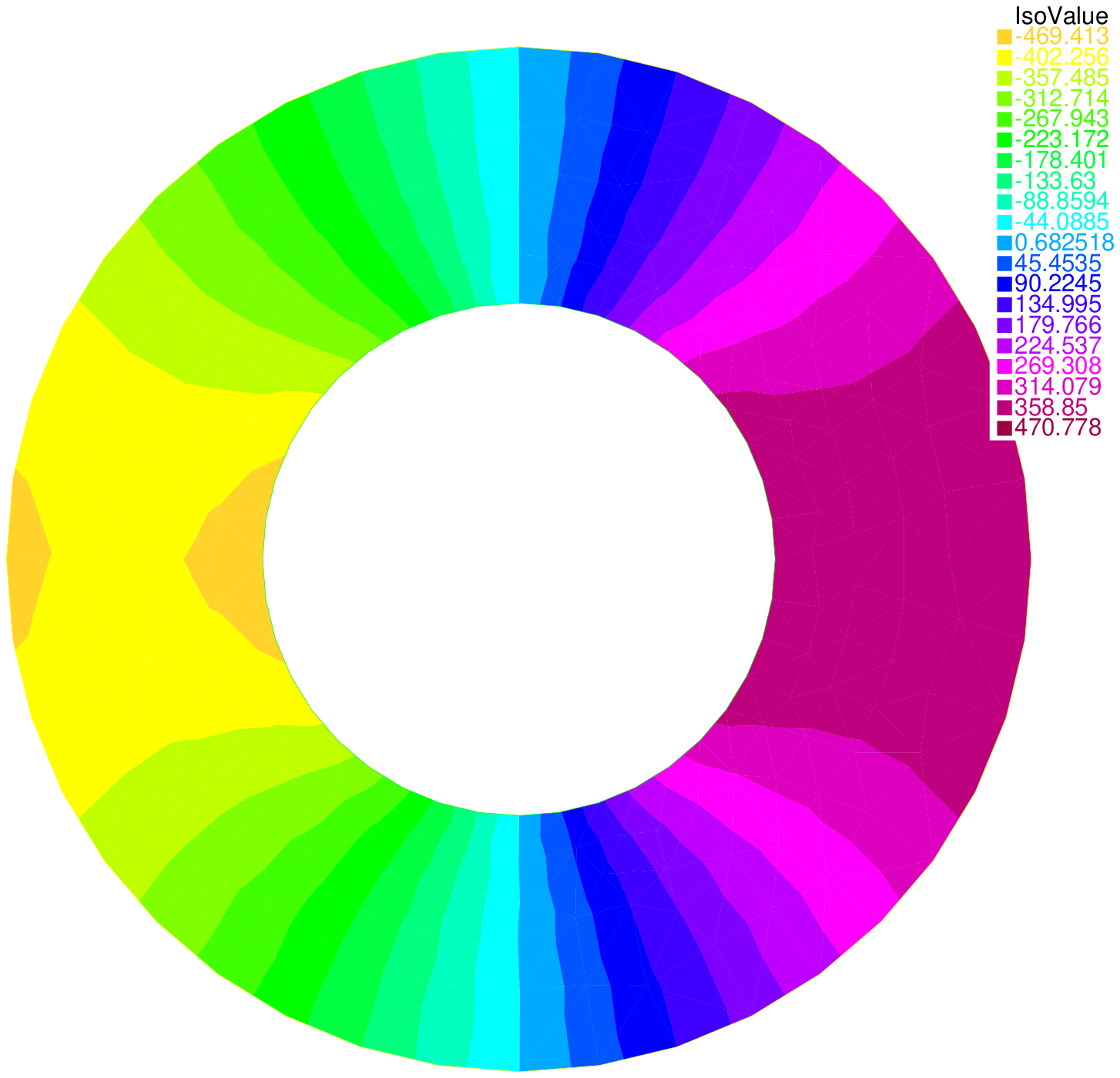}}
\subfigure[Solution in $\Omega$\label{fig2-3}]
{\includegraphics[width = 0.32\textwidth]{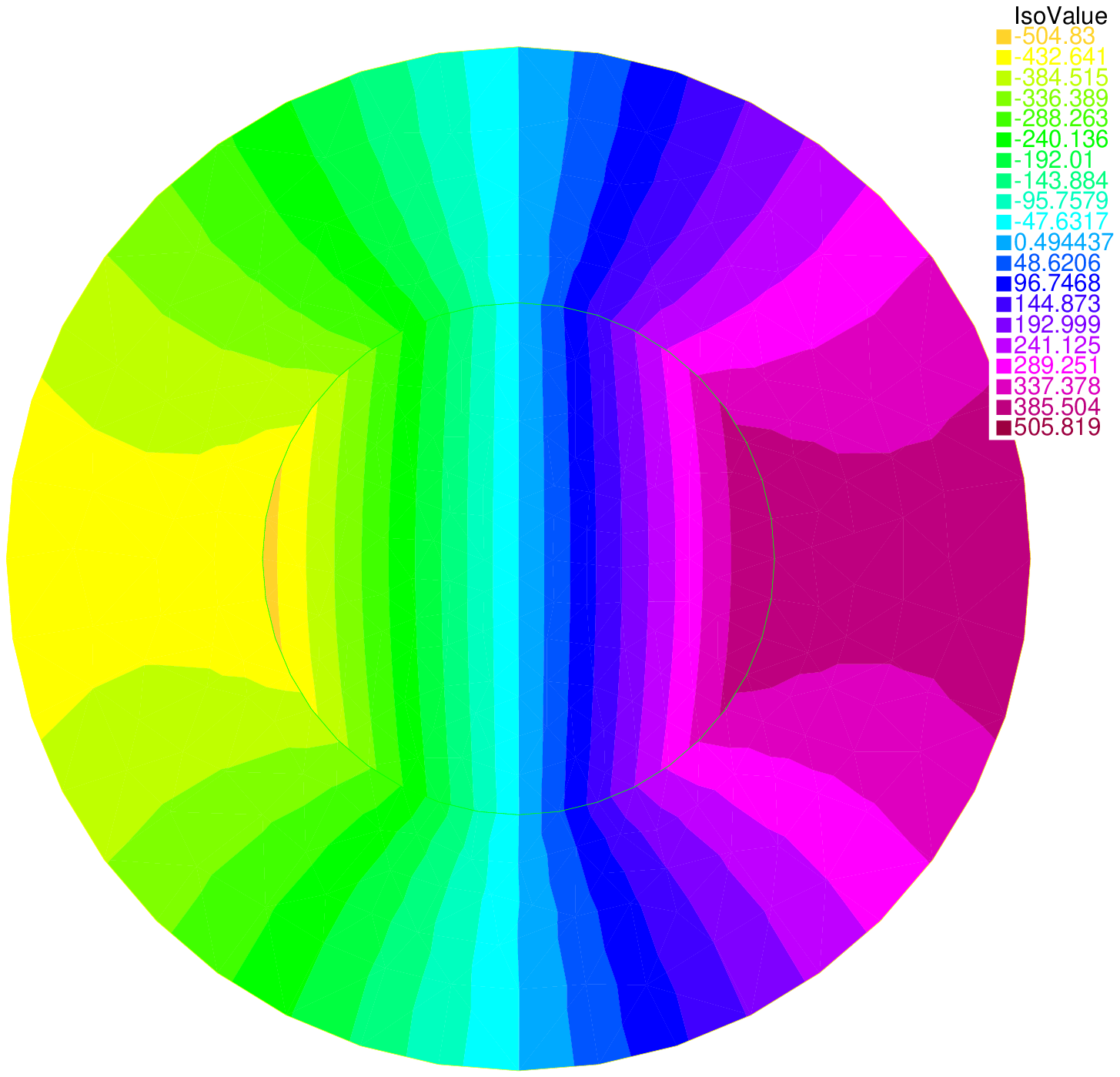}}\\[1mm]
\subfigure[Solution in $\Omega_i$\label{fig2-4}]
{\includegraphics[width = 0.32\textwidth]{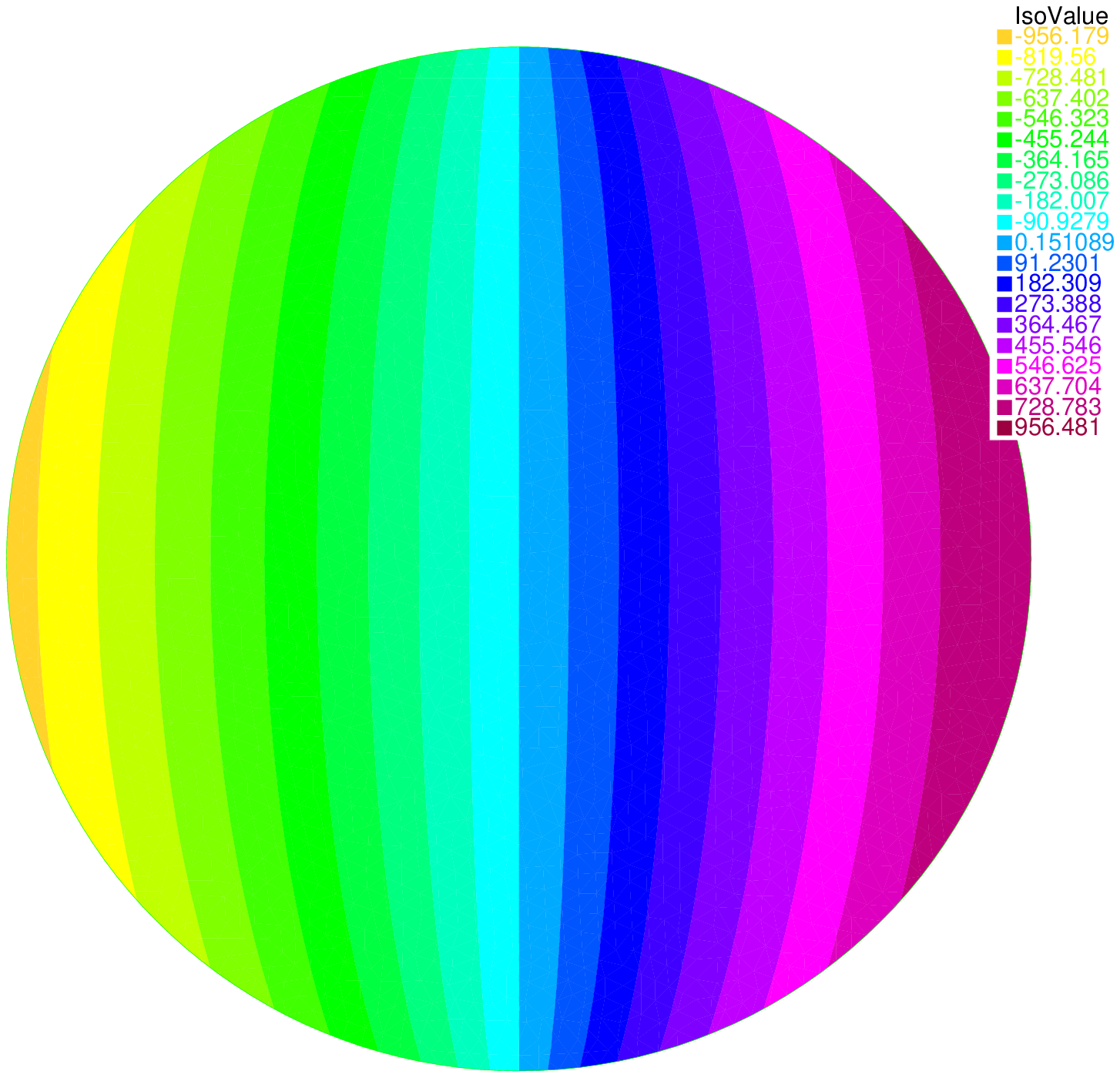}}
\subfigure[ Solution in $\Omega_e$ \label{fig2-5}]
{\includegraphics[width = 0.32\textwidth]{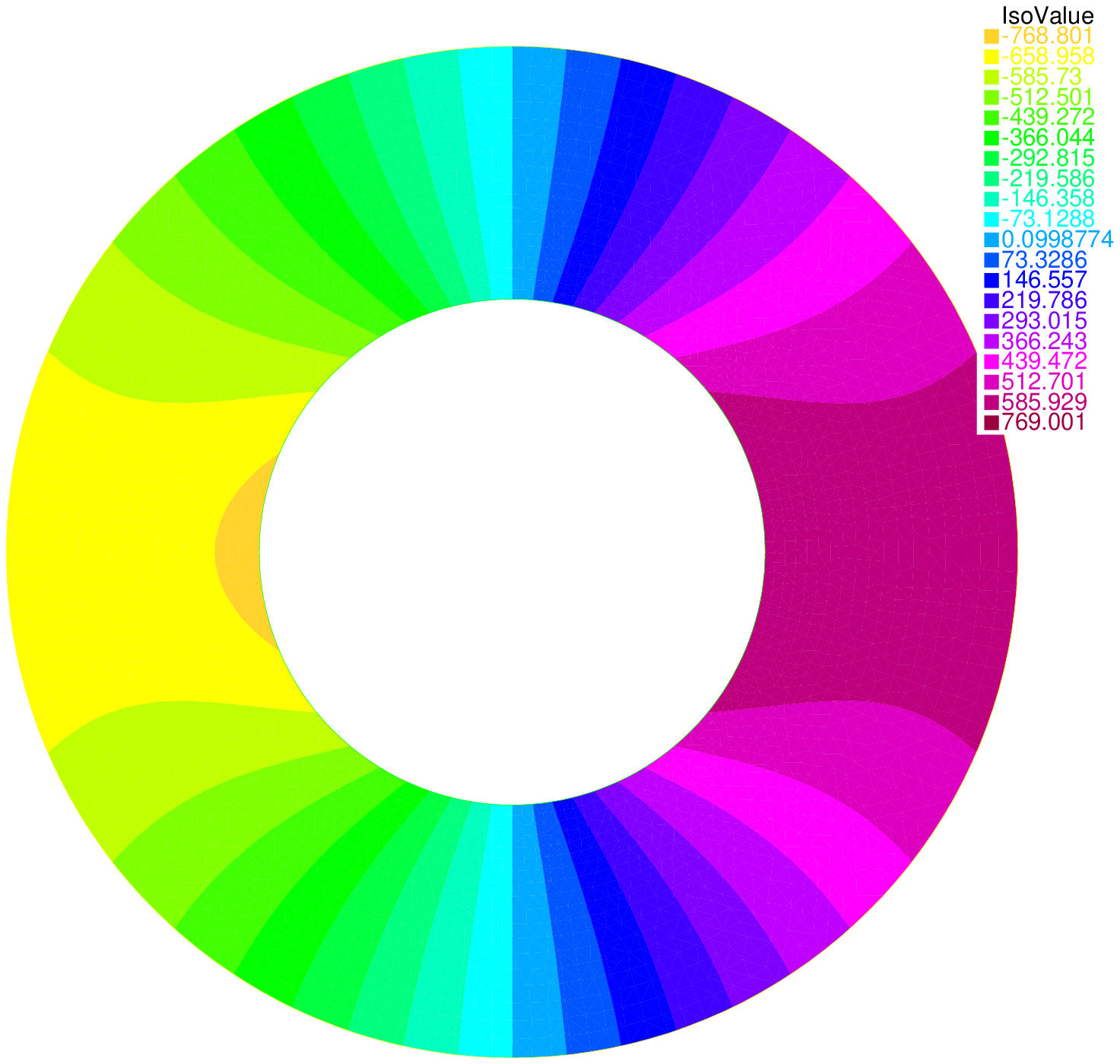}}
\subfigure[Solution in $\Omega$ \label{fig2-6}]
{\includegraphics[width = 0.32\textwidth]{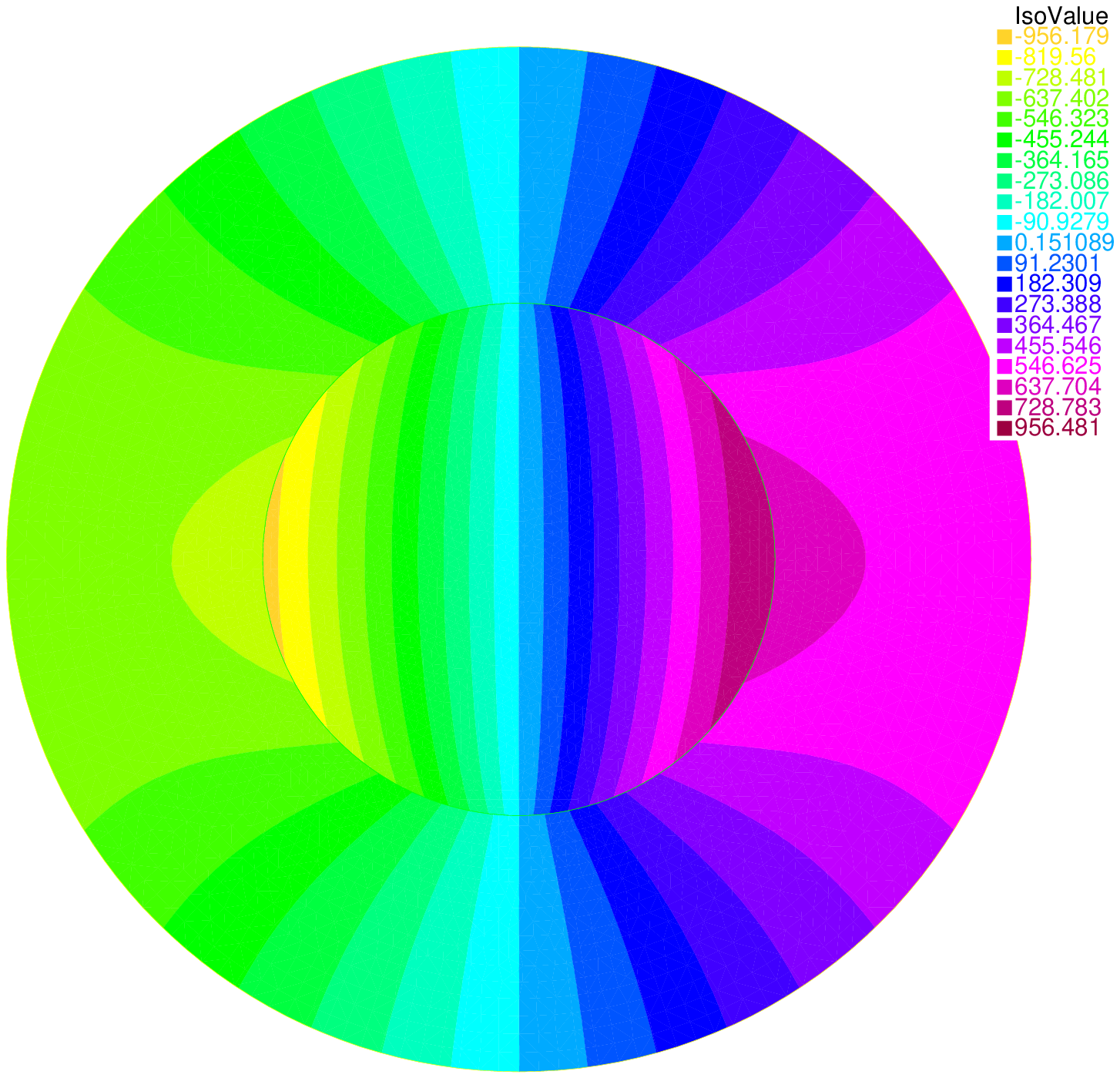}}
\end{center}
\caption{Computed solution $(u_i,u_e)$ at $T=0.5$ and $T=1.$}
\end{figure}

\subsection{Example 3}

In the third example, we take the function $s$ which has  been considered in Kavian et al. \cite{KLPW14}. It is the globally monotone function
\[
s(t)=S_L+\frac{(S_R-S_L)}{2}(1+\tanh (Ke\,(\vert t\vert-Vr))) , 
\]
where $V_r$, $Ke$, $S_L$, $S_R$ are given constants. To make the problem differentiable, we replace $\vert t\vert$ by $\sqrt{t^2+\epsilon^2}$. We consider the same boundary condition on $\d\Omega$  and a zero initial condition, like in example 2. In Figures \ref{fig3}-\ref{fig3-1one}, we plot the solution $u=(u_i,u_e)$ at times $T=0.5$ and $T=1$. The time step is $0.05$ and the 
mesh size $h_x= 0.07$. We take the constants $Ke=10$, $S_L=1.9$, $S_R=10^2$, $Vr=2.9$, and $E=1$.  

\begin{figure}[h!]
\begin{center}
\subfigure[Solution in $\O_i$\label{fig3}]
{\includegraphics[width = 0.32\textwidth]{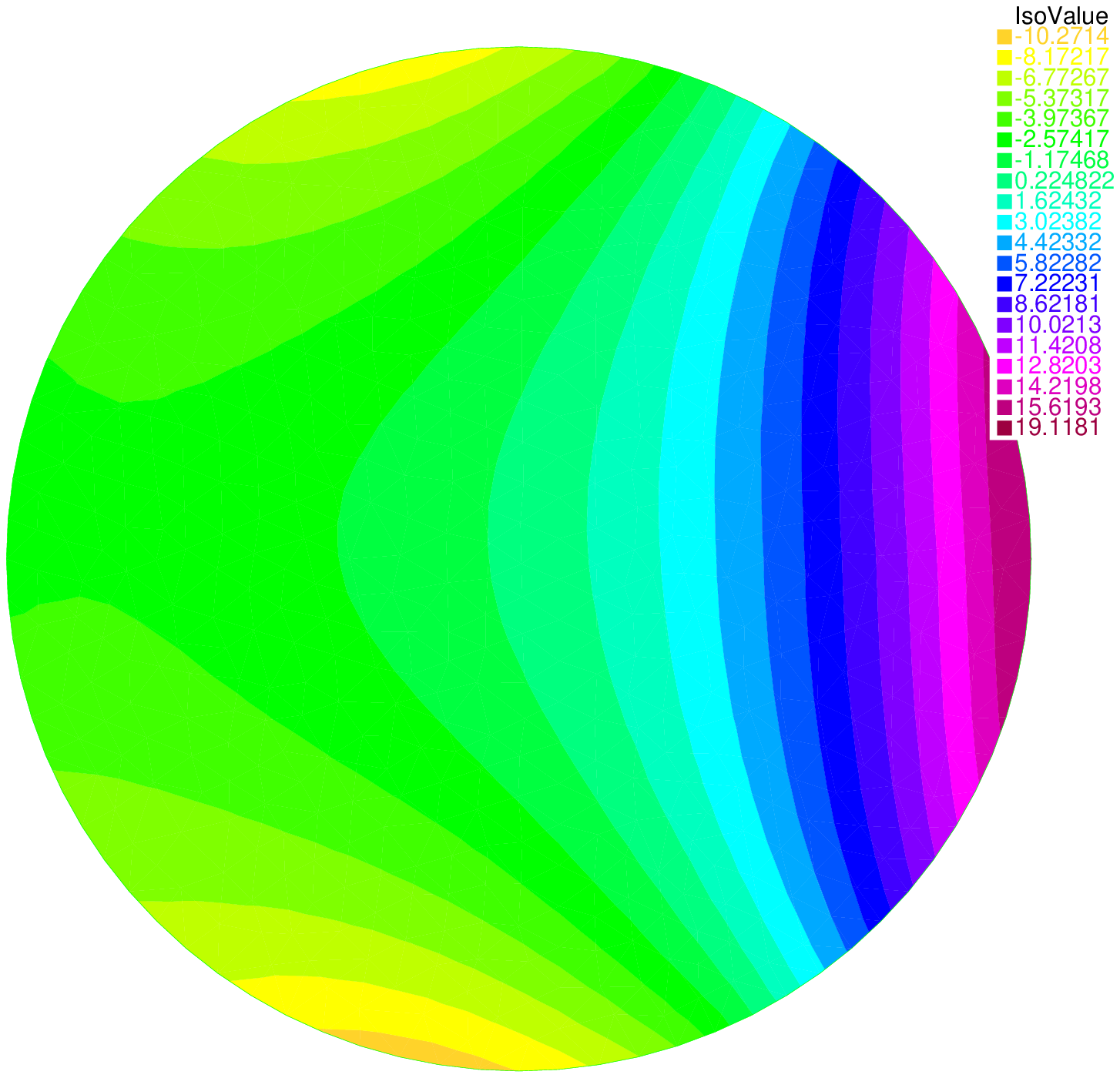}}
\subfigure[Solution in $\O_e$\label{fig3-0}]
{\includegraphics[width = 0.32\textwidth]{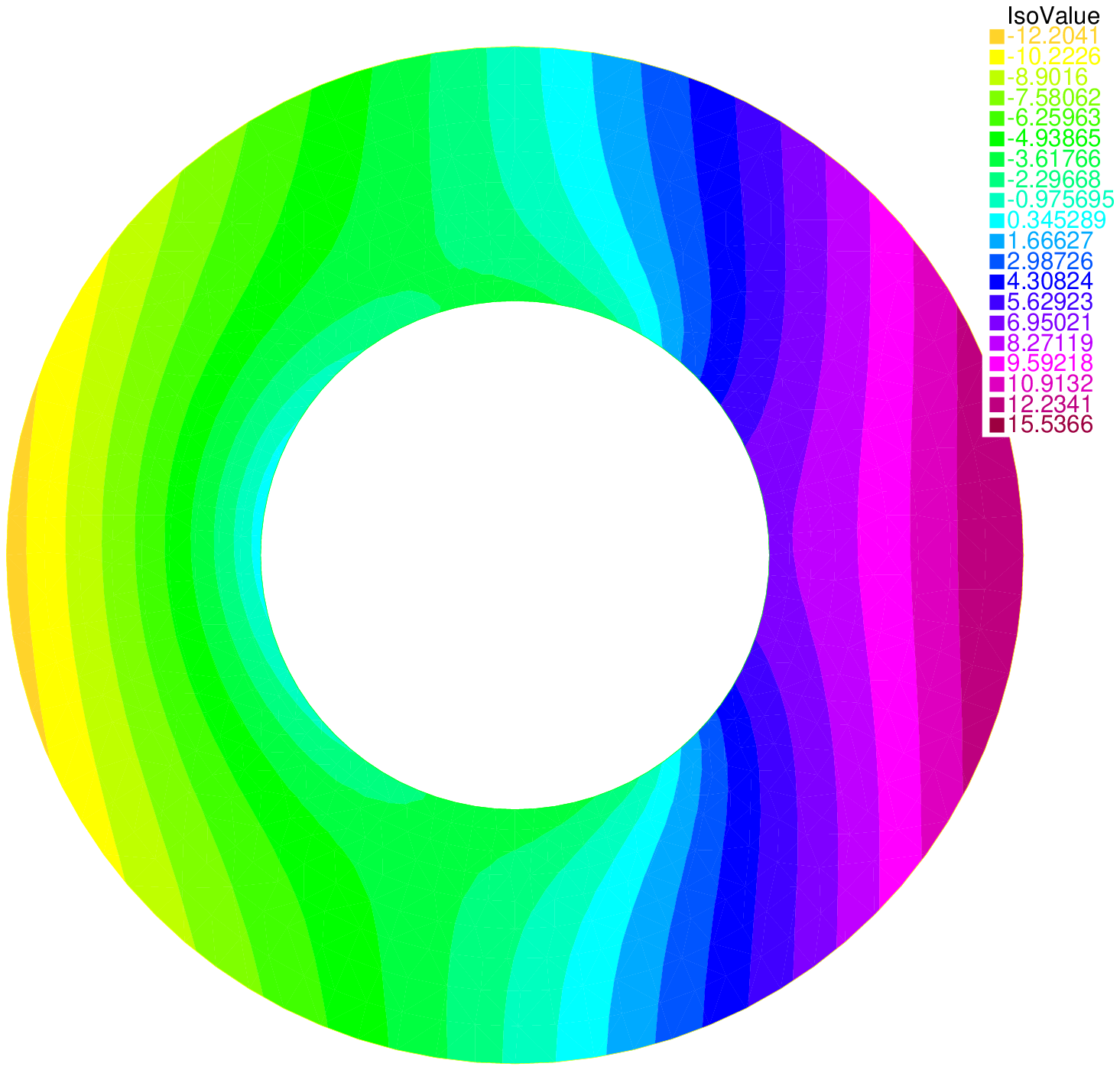}}
\subfigure[Solution in $\O$\label{fig3-1}]
{\includegraphics[width = 0.32\textwidth]{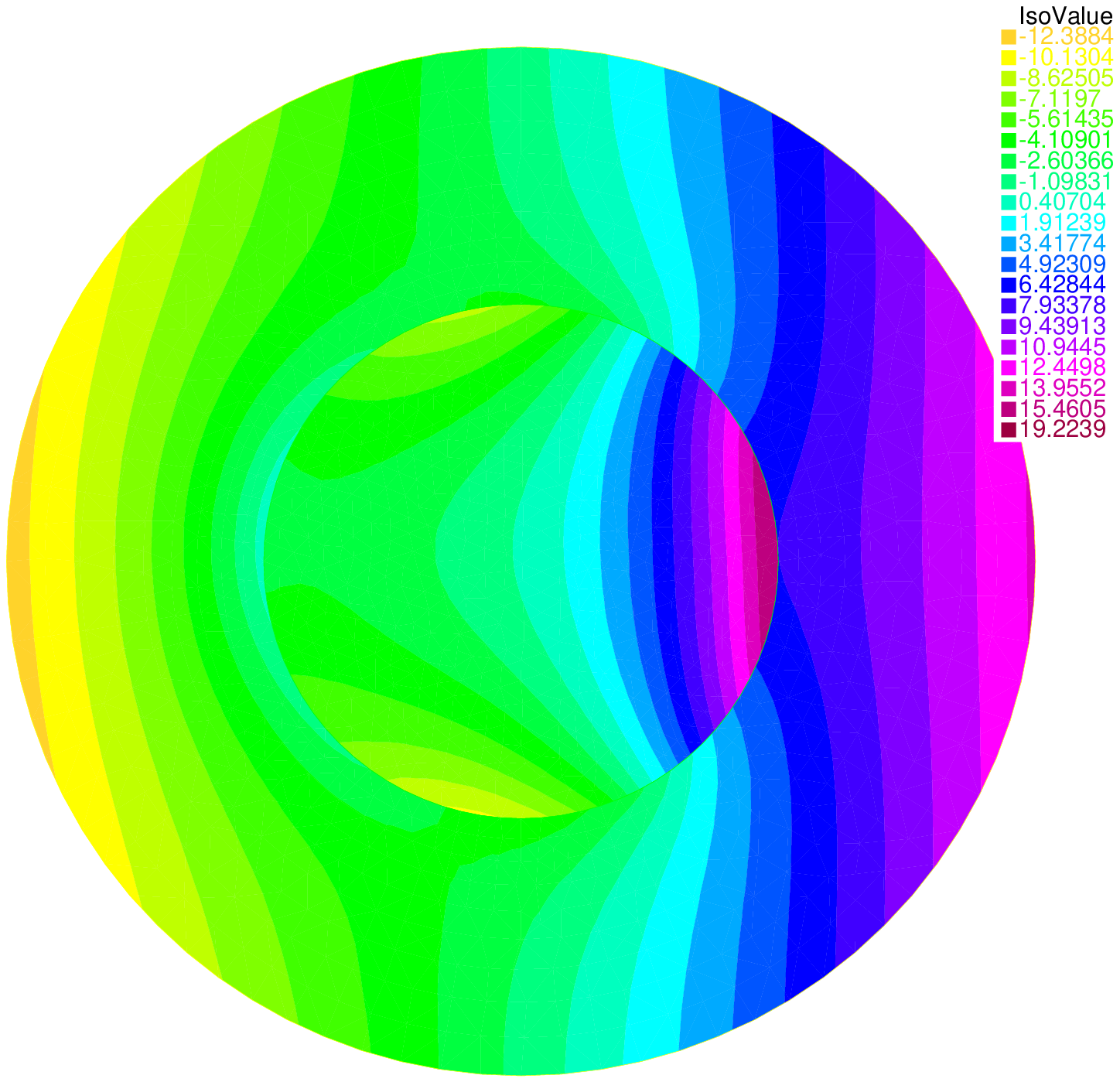}}\\[1mm]
\subfigure[Solution in $\O_i$\label{fig3one}]
{\includegraphics[width = 0.32\textwidth]{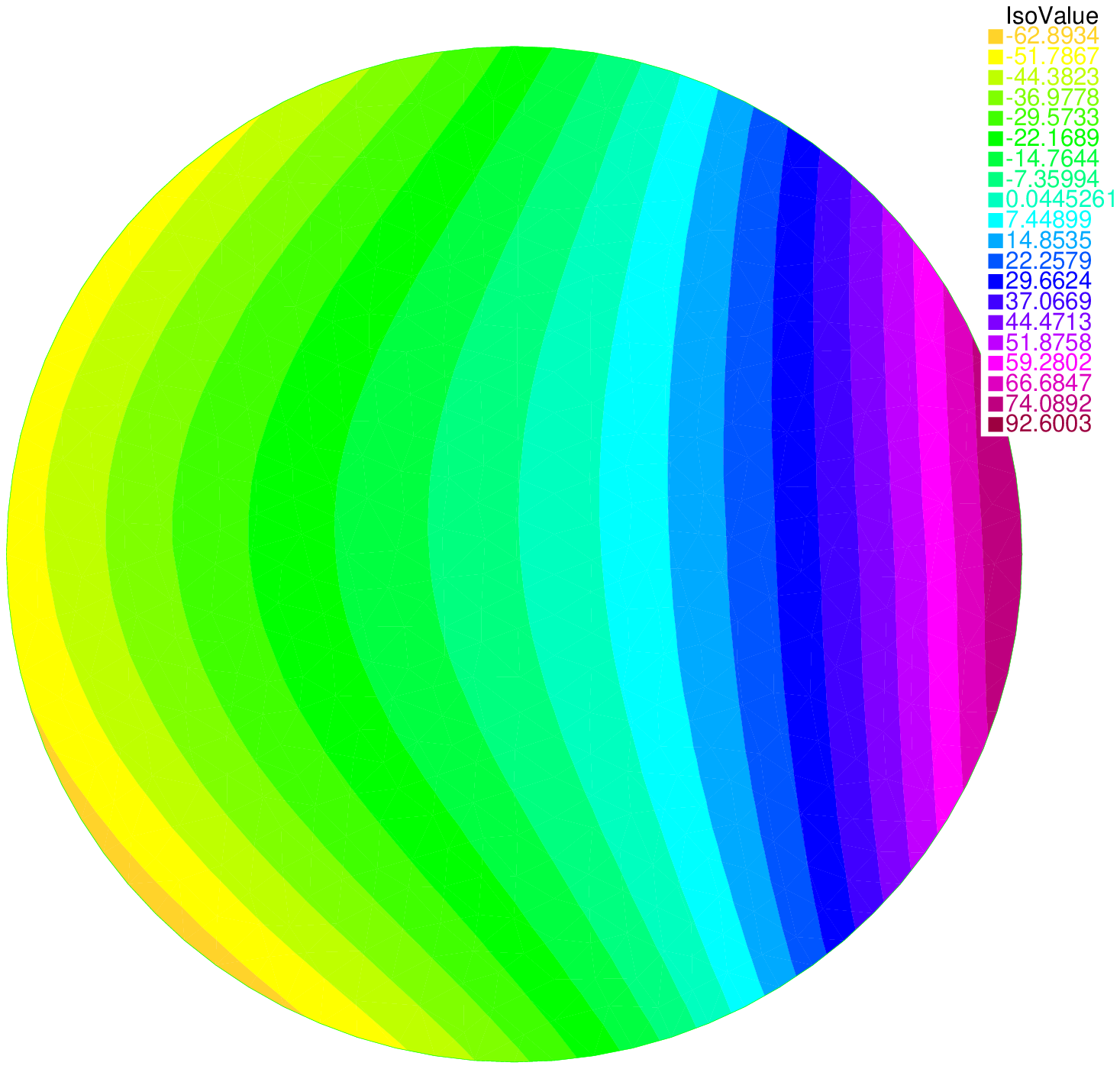}}
\subfigure[Solution in $\O_e$\label{fig3-0one}]
{\includegraphics[width = 0.32\textwidth]{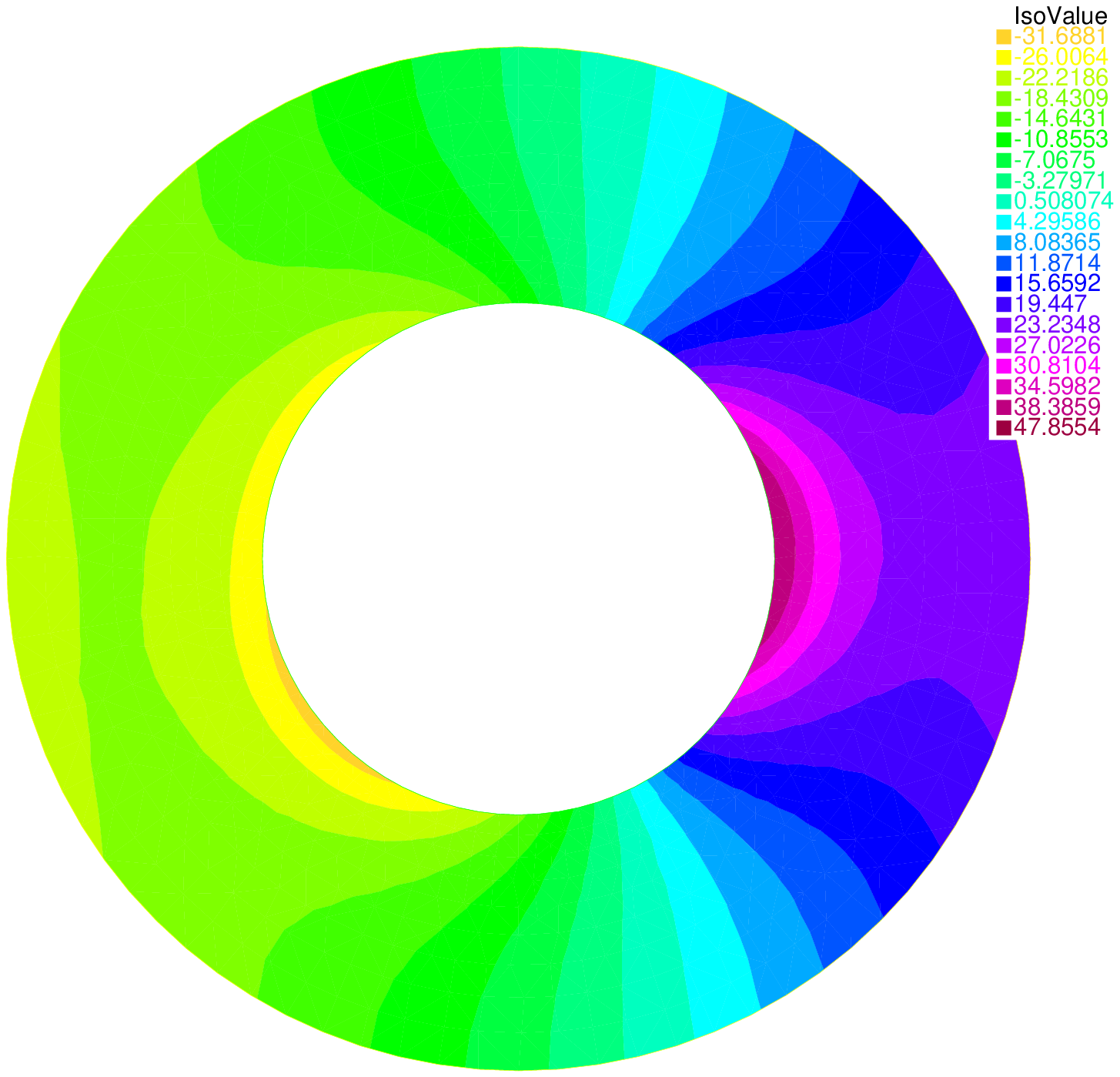}}
\subfigure[Solution in $\O$\label{fig3-1one}]
{\includegraphics[width = 0.32\textwidth]{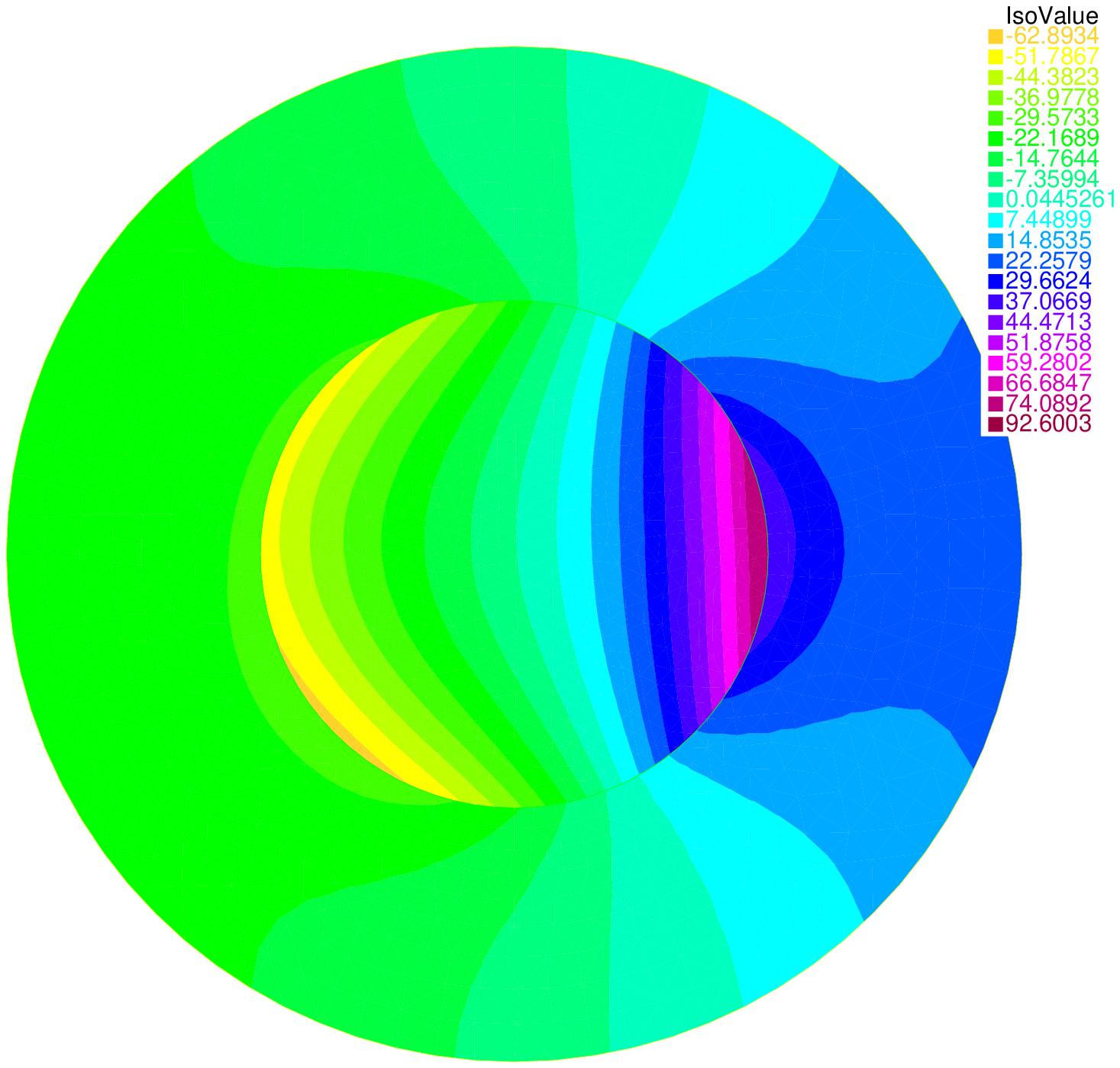}}
\end{center}
\caption{Computed solution $(u_i,u_e)$ at $T=0.5$ and $T=1$}
\end{figure}

\begin{remark}
The numerical results with two different nonlinearities $S$ are in this example quite similar, since the two functions ensure a transmission/transition from the left state characterized by the potential $S_L$ to the right state $S_R$. The main difference is the smoothness of the transition from the left to the right. Note also the role of the constants $V_r$ and $K_e$ on the profile of this transition for the example 3, which has no counterpart in the example 2 even if $\epsilon$ tends to sharpen the profile. We emphasize that our main concern in this article is the possibility of using several kind of nonlinearities, geometries etc. in the framework of the $j$-gradient theory and not to validate any choice of the electropermeabilisation model.
\end{remark}

We end this numerical section by considering the data in a range close to the physical parameters, namely $Ke=10$,$S_L=1.9$, $S_R=10^6$, $Vr=1.5$, and $E=4$. The results are plotted in Figures \ref{fig3b}-\ref{fig3b-2} at time $T=0.5$. One may observe that high and fast variations of the electrical potential are located close to $\Gamma$.
\begin{figure}[h!]

\begin{center}
\subfigure[Solution $u_i$, $T=0.5$ \label{fig3b}]
{\includegraphics[width = 0.32\textwidth]{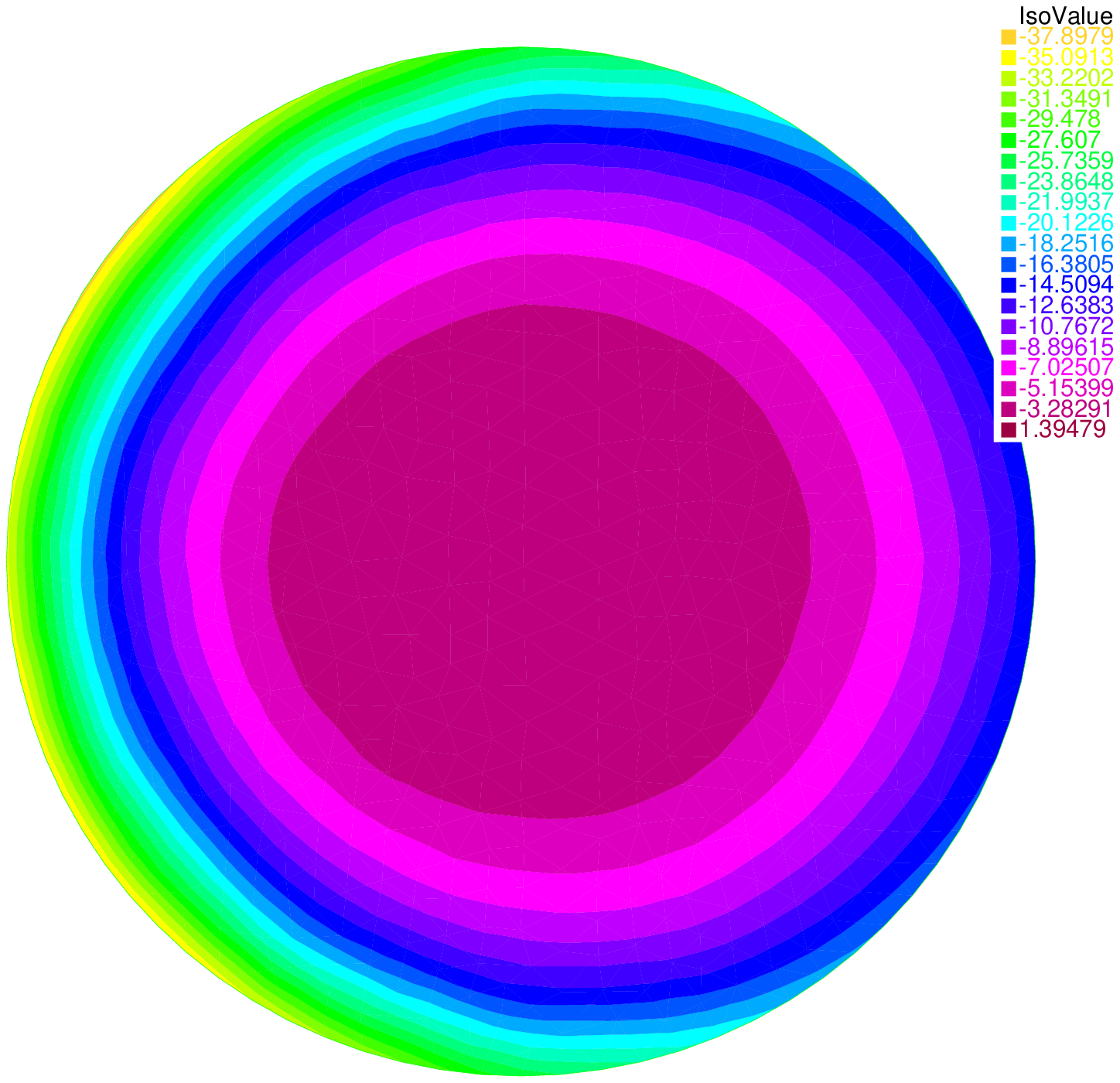}}
\subfigure[Solution $u_e$, $T=0.5$ \label{fig3b-1}]
{\includegraphics[width = 0.32\textwidth]{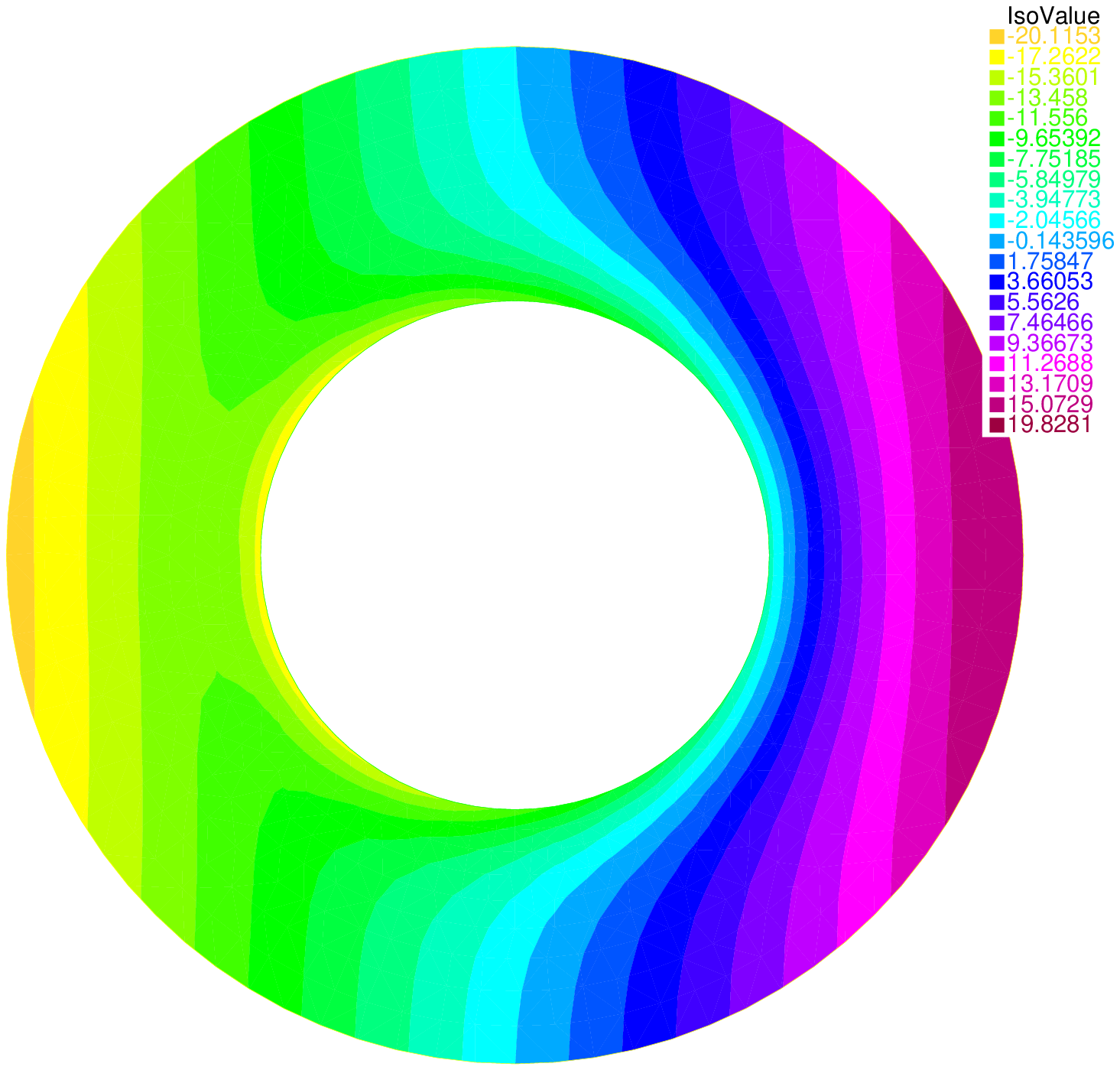}}
\subfigure[Solution $u$, $T=0.5$ \label{fig3b-2}]
{\includegraphics[width = 0.32\textwidth]{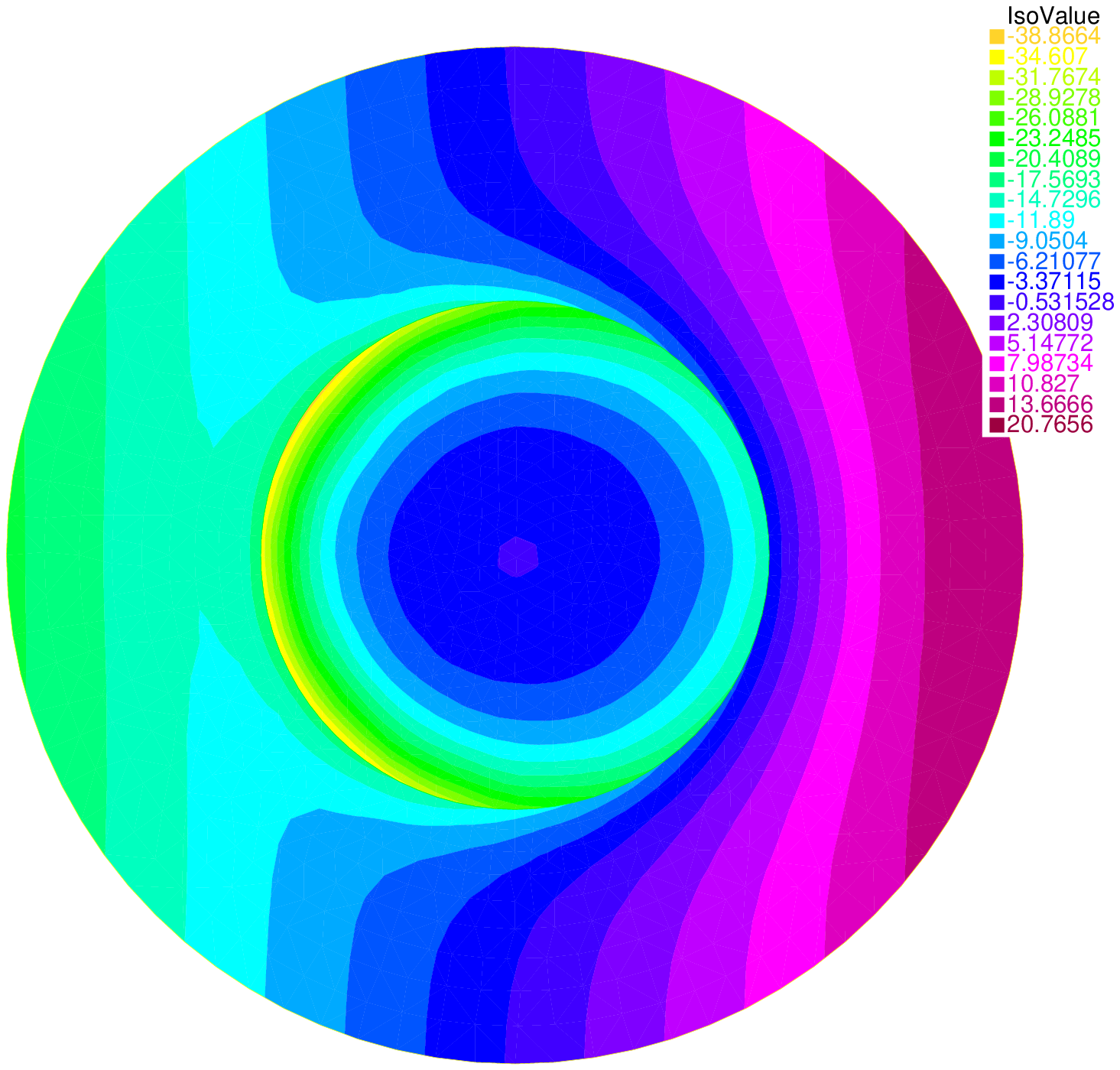}}
\end{center}
\caption{Solution with nearly physical parameters}
\end{figure}

The last results correspond to the example 3 in the sense that we take the same nonlinearity $s$, but for different geometries. We have set $Ke=10$, $S_L=1.9$, $S_R=10^3$, $Vr=2.01$, and $E=1$. It is interesting to note how the shape of $\Gamma$ changes the solution, namely both the profile and the magnitude.

\begin{figure}[h!]

\begin{center}
\subfigure[Cassini egg mesh\label{fig4-1}]
{\includegraphics[width = 0.35\textwidth]{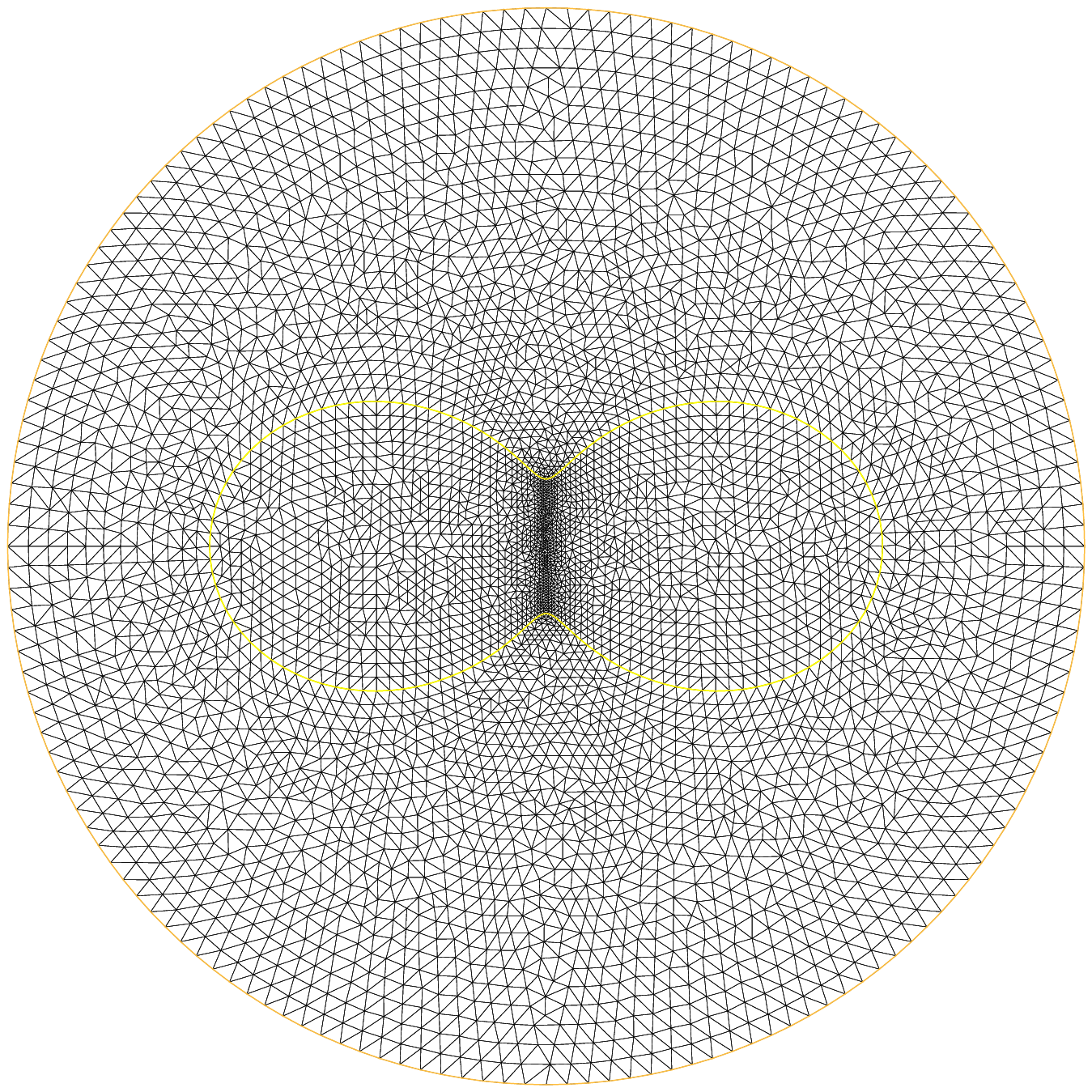}}
\subfigure[Solution in $\Omega$ \label{fig4-2}]
{\includegraphics[width = 0.35\textwidth]{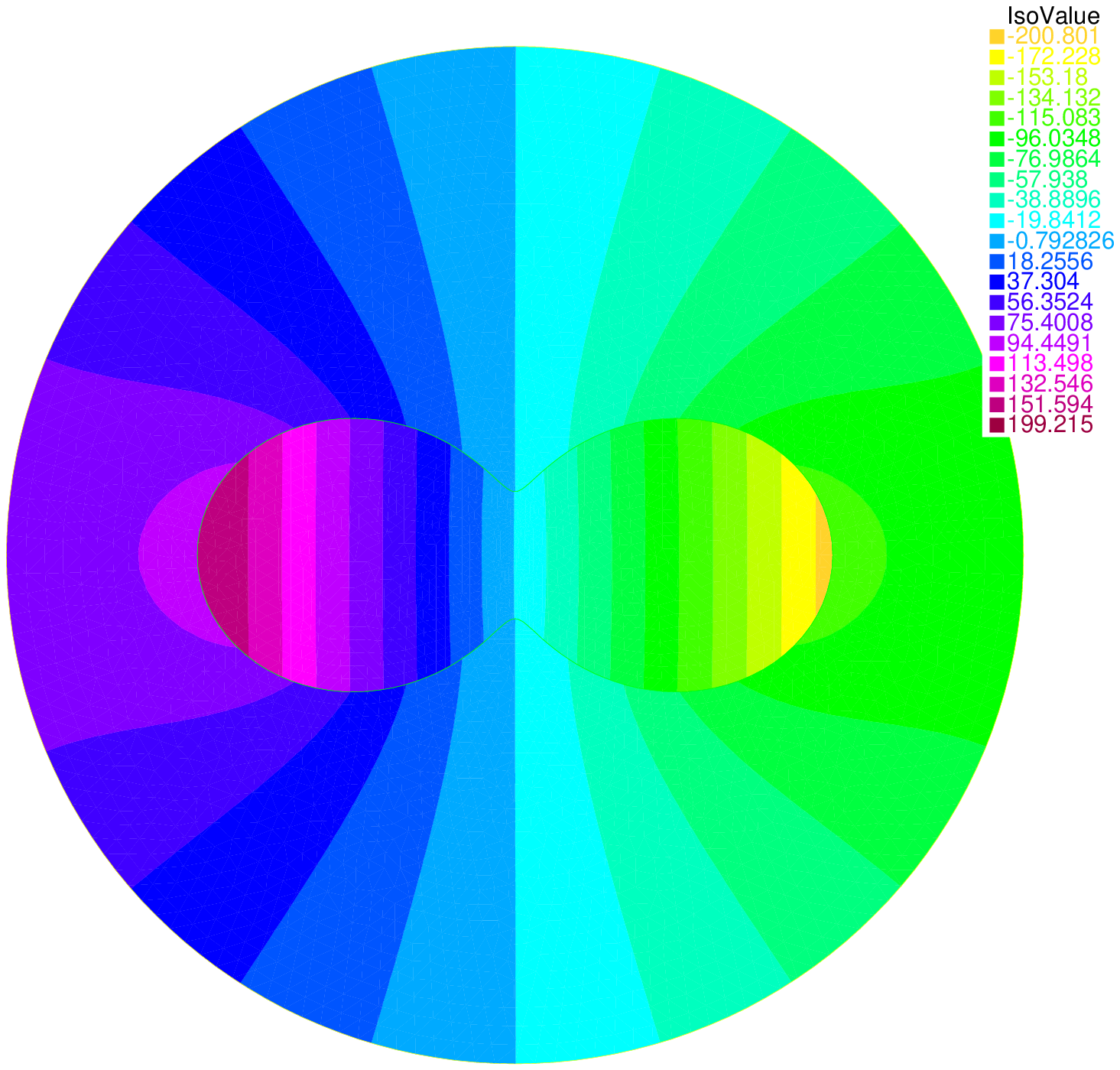}}\\[1mm]
\subfigure[the solution  $u_e$\label{fig4-3}]
{\includegraphics[width = 0.32\textwidth]{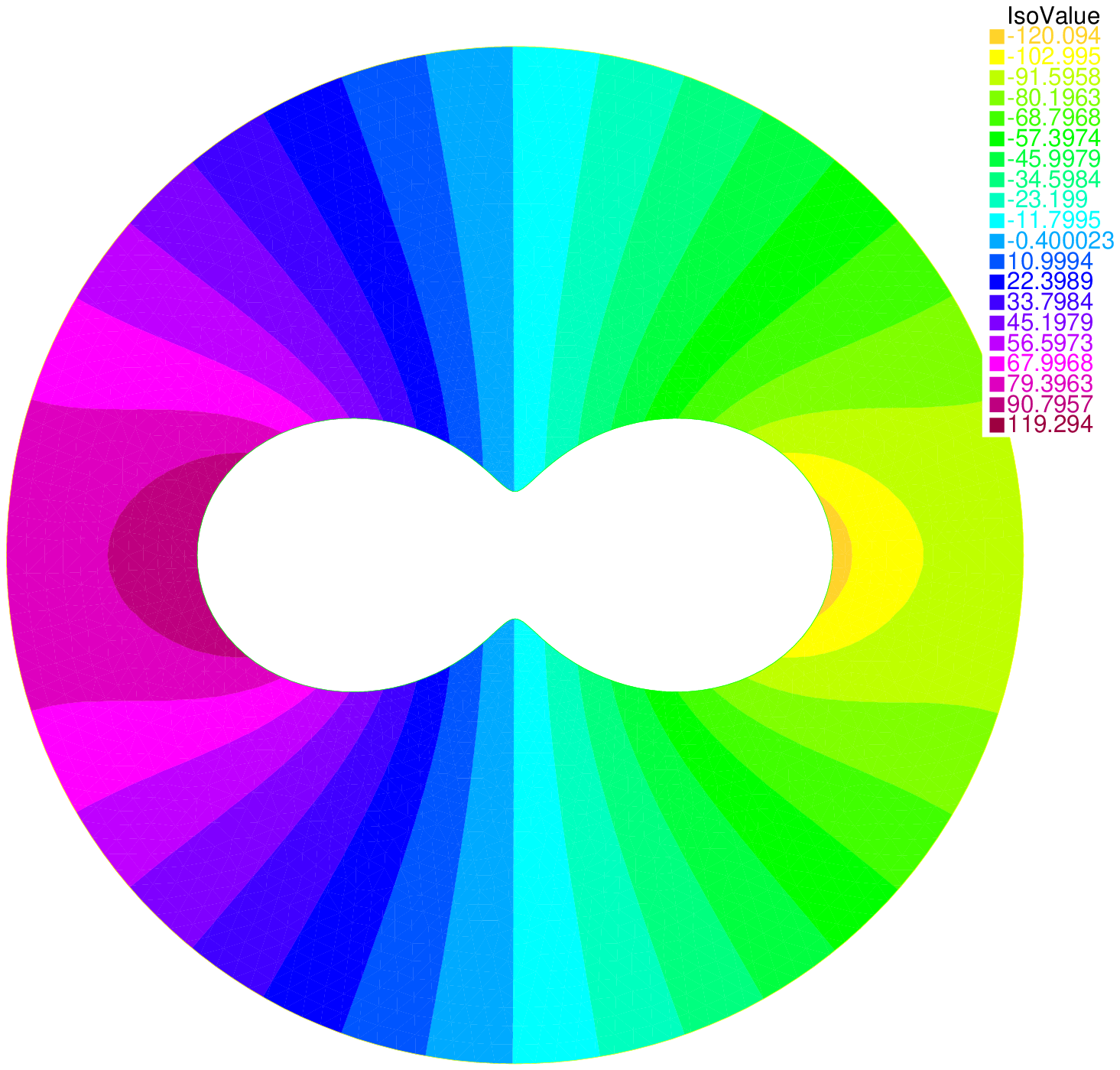}}
\subfigure[ the solution $u_i$ \label{fig4-4}]
{\includegraphics[width = 0.32\textwidth]{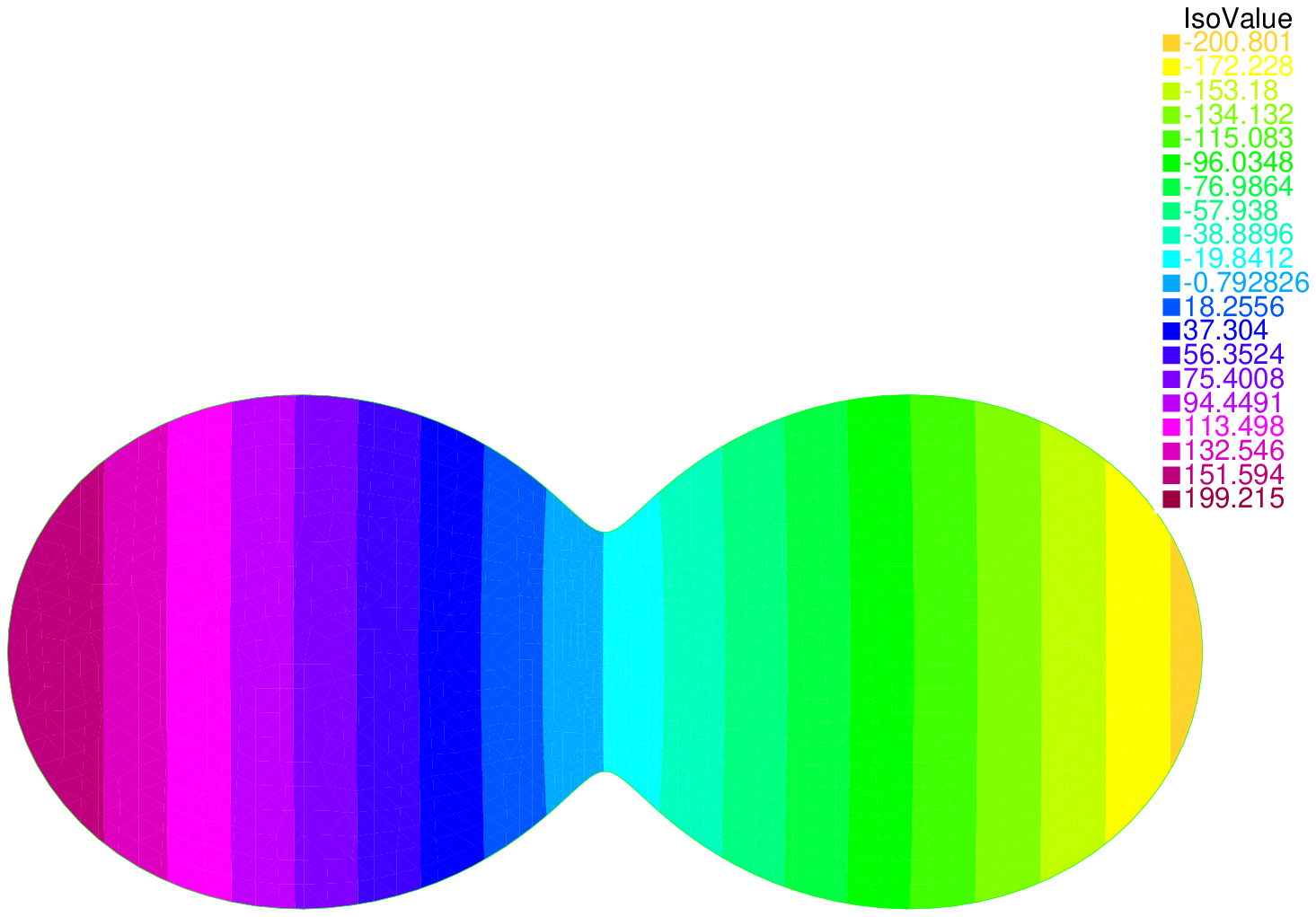}}
\end{center}
\caption{Cassini egg shape}
\end{figure}

\begin{figure}[h!]
\begin{center}
\subfigure[Snale mesh  \label{fig4-1snail}]
{\includegraphics[width = 0.35\textwidth]{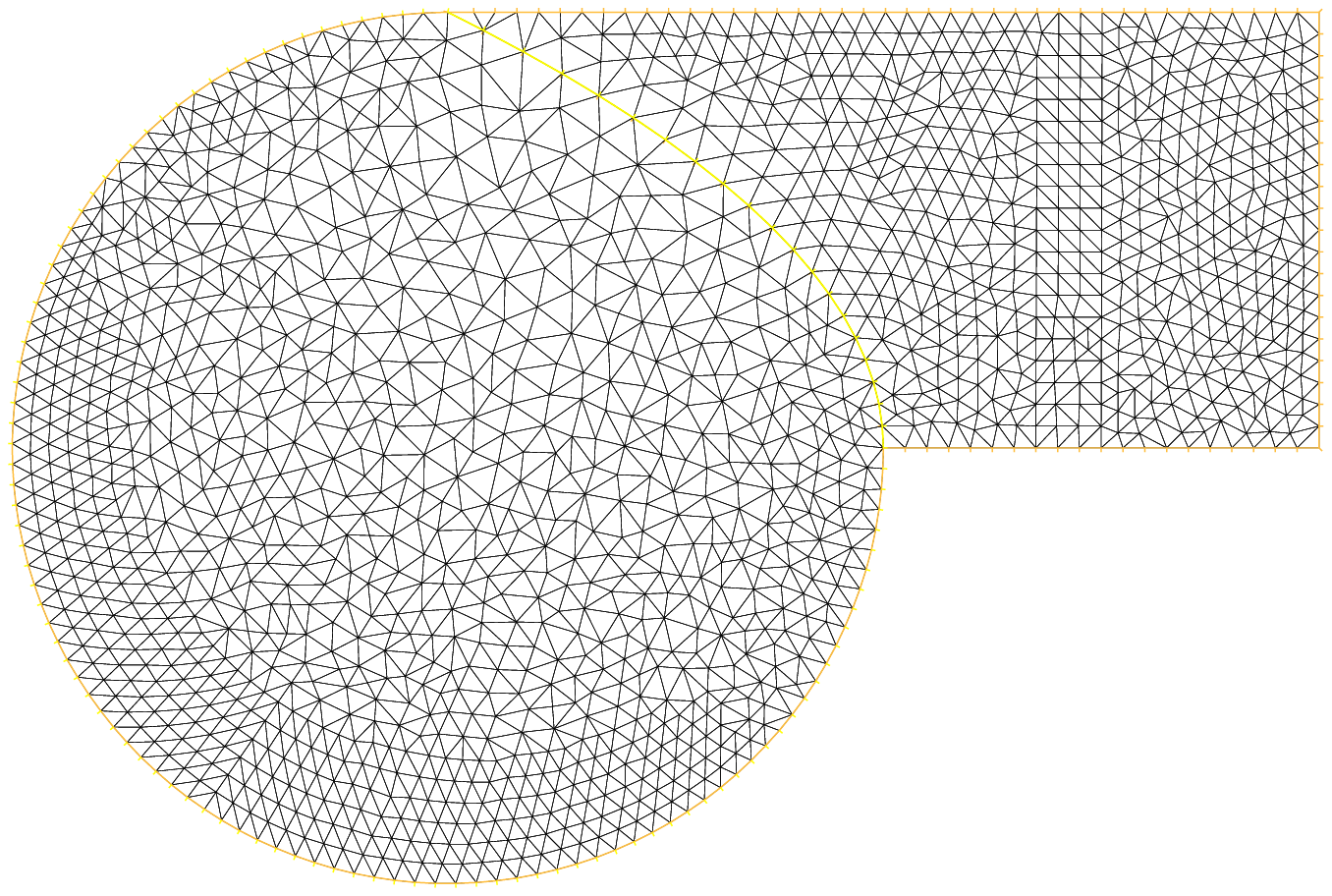}}
\subfigure[Solution in $\Omega$ \label{fig4-2snail}]
{\includegraphics[width = 0.35\textwidth]{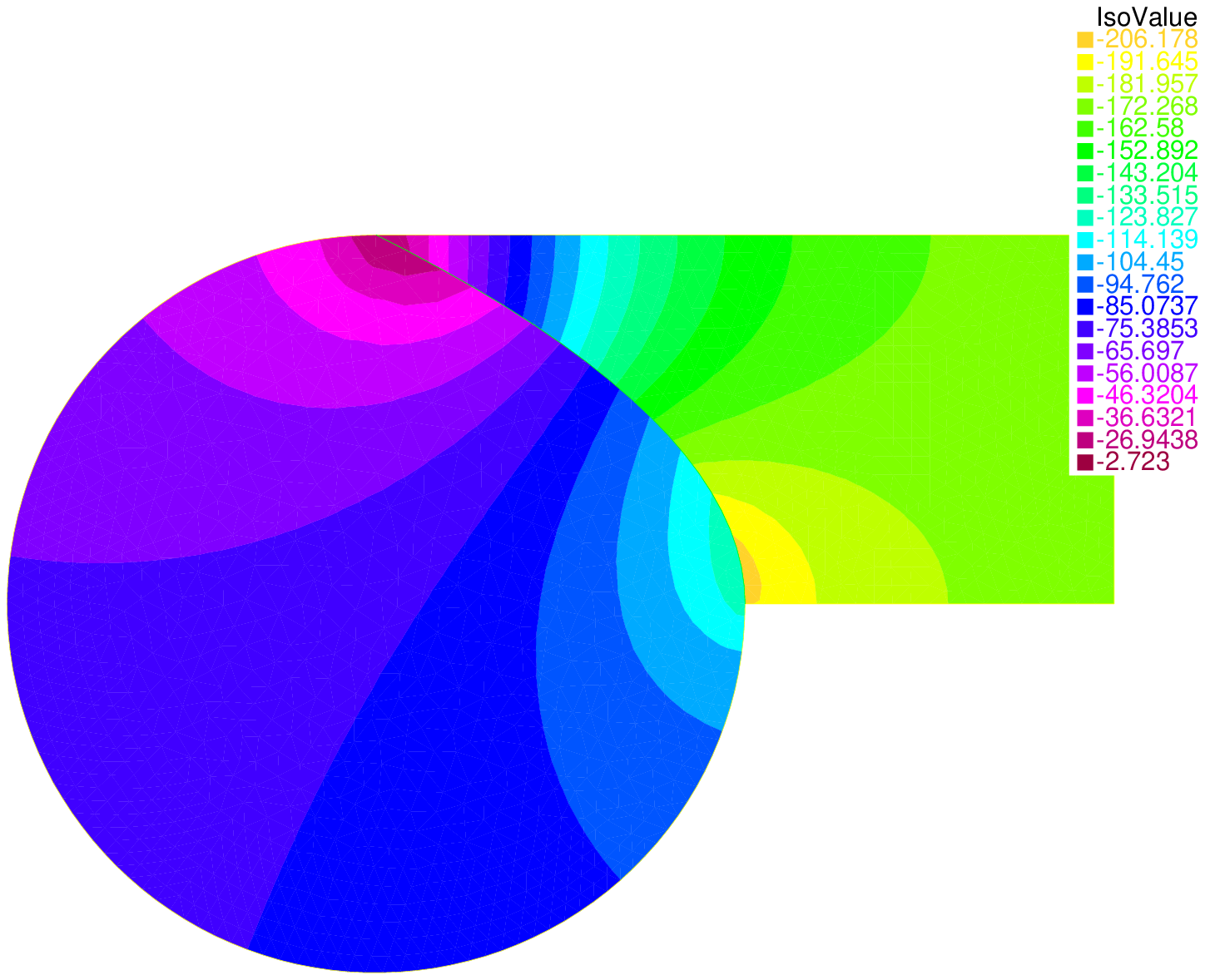}}\\[1mm]
\subfigure[the solution  $u_i$ at $T=0.025$\label{fig4-3snail}]
{\includegraphics[width = 0.32\textwidth]{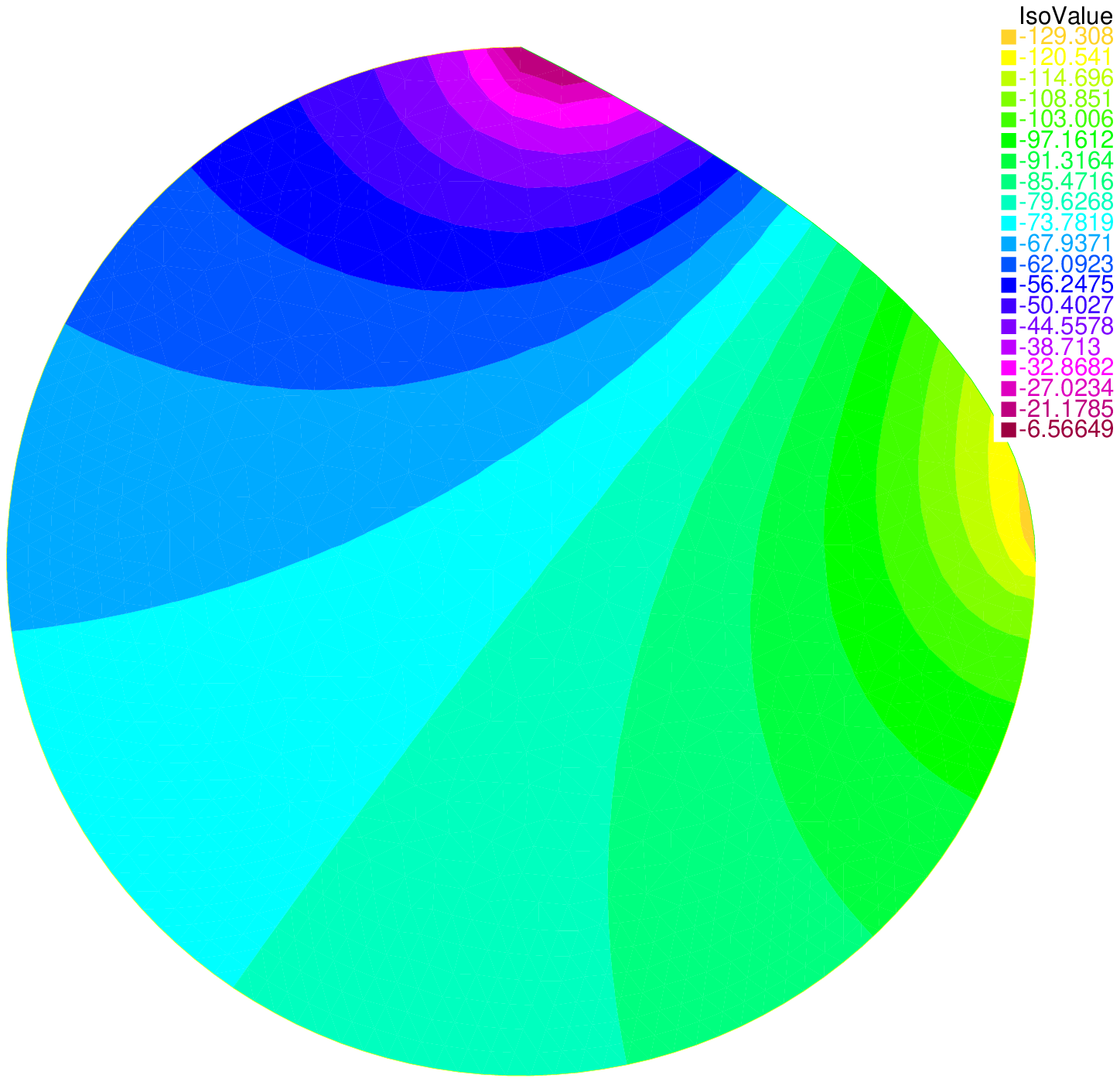}}
\subfigure[ the solution $u_e$ at $T=0.025$\label{fig4-4snail}]
{\includegraphics[width = 0.32 \textwidth]{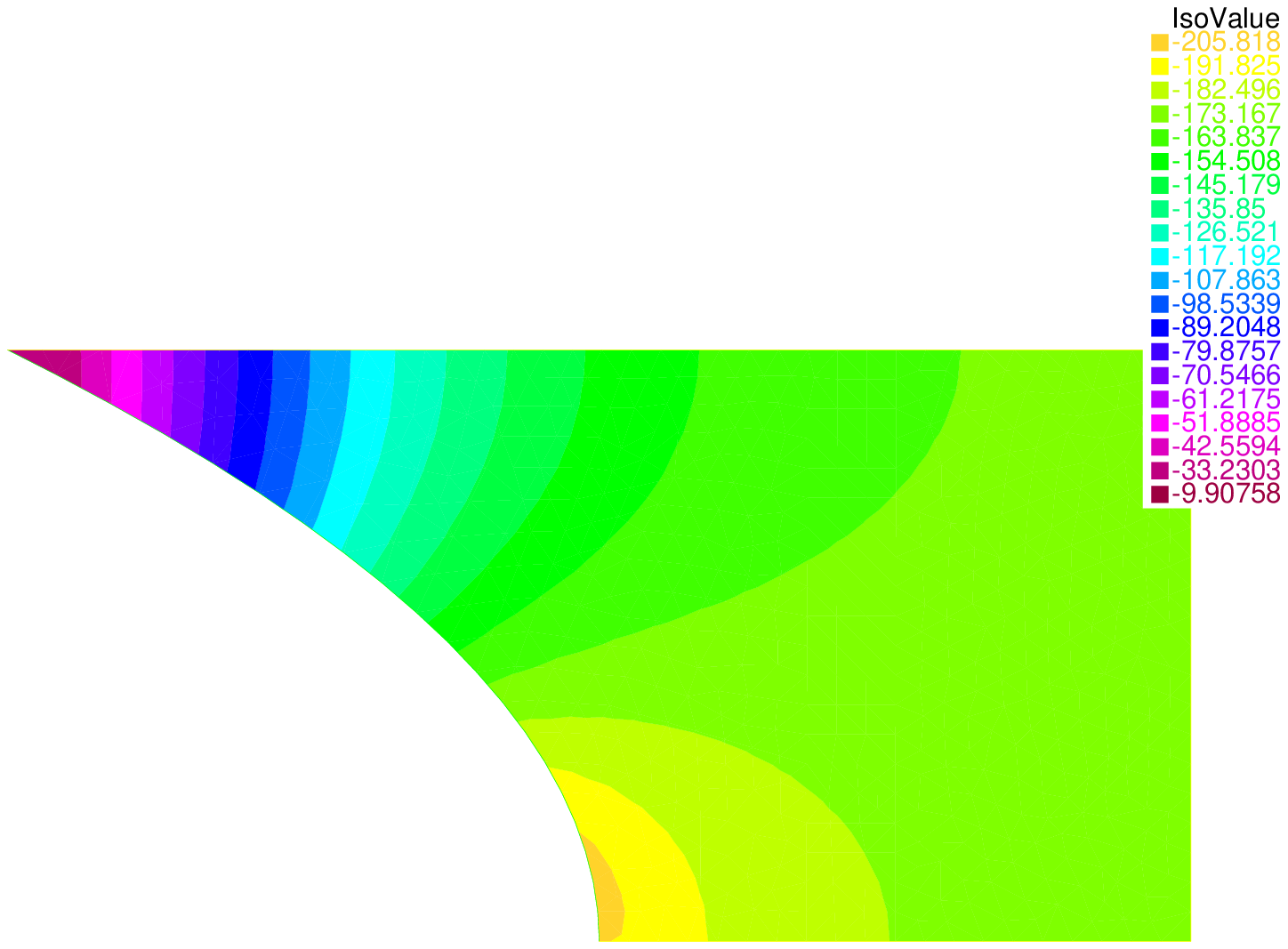}}
\end{center}
\caption{A snale cell}
\end{figure}

\newpage

\providecommand{\bysame}{\leavevmode\hbox to3em{\hrulefill}\thinspace}

\bibliographystyle{amsplain}
%\begin{thebibliography}{10}

\def\cprime{$'$} 
  \def\ocirc#1{\ifmmode\setbox0=\hbox{$#1$}\dimen0=\ht0 \advance\dimen0
  by1pt\rlap{\hbox to\wd0{\hss\raise\dimen0
  \hbox{\hskip.2em$\scriptscriptstyle\circ$}\hss}}#1\else {\accent"17 #1}\fi}
  \def\cprime{$'$} \def\cprime{$'$} \def\cprime{$'$}
\providecommand{\bysame}{\leavevmode\hbox to3em{\hrulefill}\thinspace}
\providecommand{\MR}{\relax\ifhmode\unskip\space\fi MR }
\providecommand{\MRhref}[2]{%
  \href{http://www.ams.org/mathscinet-getitem?mr=#1}{#2}
}
\providecommand{\href}[2]{#2}

\end{document}